\documentclass[reqno]{amsart}
\usepackage{amssymb,latexsym}
\usepackage{amsmath,color}
\usepackage{amsthm}
\usepackage{epsfig}
\usepackage{graphicx}
\usepackage{titletoc}
\usepackage{palatino,mathpazo}
\numberwithin{equation}{section}
\newtheorem{theorem}{Theorem}[section]

\newtheorem{lemma}[theorem]{Lemma}
\newtheorem{remark}[theorem]{Remark}

\newtheorem{definition}{Definition}[section]

\newcommand{\prn}{{\textbf{(}\mathbf P_{\rho_n}\textbf{)}}}
\newcommand{\lf}{L^\infty(\Omega)}

\newcommand{\ov}[1]{\overline{#1}}

\theoremstyle{definition}

\def\XXint#1#2#3{{\setbox0=\hbox{$#1{#2#3}{\int}$}
     \vcenter{\hbox{$#2#3$}}\kern-.5\wd0}}

\def\om{\overline{\Omega}}

\newcommand{\graf}[1]{\left\{\begin{array}{ll}#1\end{array}\right.}

\newcommand{\sg}{\sigma}
\newcommand{\al}{\alpha}
\newcommand{\gm}{\gamma}

\newcommand{\lm}{\lambda}

\begin{document}
\title[Mean field equations with singular data]{Local uniqueness and non-degeneracy of blow up solutions of mean field equations with singular data}

\author[D. Bartolucci]{Daniele Bartolucci}
 \address{Daniele Bartolucci, Department of Mathematics, University of Rome {\it "Tor Vergata"},  Via della ricerca scientifica n.1, 00133 Roma,
Italy.}
\email{bartoluc@mat.uniroma2.it}

\author[A. Jevnikar]{Aleks Jevnikar}
 \address{Aleks Jevnikar, Scuola Normale Superiore, Piazza dei Cavalieri 7, 56126, Pisa, Italy.}
\email{aleks.jevnikar@sns.it}

\author[Y. Lee]{Youngae Lee}
\address{Youngae ~Lee, Department of Mathematics Education, Teachers College, Kyungpook National University, Daegu, South Korea}
\email{youngaelee@knu.ac.kr}

\author[W. Yang]{Wen Yang}
\address{Wen ~Yang, Wuhan Institute of Physics and Mathematics, Chinese Academy of Sciences, P.O. Box 71010, Wuhan 430071, P. R. China}
\email{wyang@wipm.ac.cn}


\thanks{D. Bartolucci is partially supported by FIRB project "{\em
Analysis and Beyond}",  by PRIN project 2012, ERC PE1\_11,
"{\em Variational and perturbative aspects in nonlinear differential problems}", and by the Consolidate the Foundations
project 2015 (sponsored by Univ. of Rome "Tor Vergata"),  ERC PE1\_11,
"{\em Nonlinear Differential Problems and their Applications}".
Y. Lee is partially supported by the National Research Foundation of Korea (NRF) grant funded by the Korea government (MSIT) (No. NRF-2018R1C1B6003403).
W. Yang is partially supported by NSFC No.11801550.}

\begin{abstract}
We are concerned with the mean field equation with singular data on bounded domains. Under suitable non-degeneracy conditions we prove local uniqueness and non-degeneracy of bubbling solutions blowing up at singular points. The proof is based on sharp estimates for bubbling solutions of singular mean field equations and  suitably defined Pohozaev-type identities.  
\end{abstract}
\maketitle

\section{Introduction} \label{sec:intro}
We are concerned with a sequence of solutions of the following mean field equation with singular data
\begin{equation*}
\begin{cases}
-\Delta u_n=\rho_n\dfrac{he^{u_n}}{\int_{\Omega}he^{u_n}}\quad &\mathrm{in}~\Omega,\\
u_n=0~&\mathrm{on}~\partial\Omega,
\end{cases}
\eqno \prn
\end{equation*}
where $\Omega\subset\mathbb{R}^2$ is a smooth bounded domain, $h=h_*\exp(-4\pi\sum_{i=1}^N\alpha_iG(x,p_i))$, $p_i$ are distinct points in $\Omega$, $\alpha_i\in(0,\infty)\setminus\mathbb{N}$, $h_*\in C^{\infty}(\overline \Omega)$, and $G$ is the Green function satisfying
\begin{align*}
\begin{cases}
-\Delta G(x,p)=\delta_p\quad&\mathrm{in}~\Omega,\\
G(x,p)=0~&\mathrm{on}~\partial\Omega.
\end{cases}
\end{align*}
 
\

The mean field equation $\prn$ (and its counterpart on compact surfaces) have been widely discussed in the last decades 
because of their several applications in
Mathematics and Physics, such as Electroweak and Chern-Simons self-dual vortices \cite{sy2,T0,yang},
conformal metrics on surfaces with \cite{Troy} or without conical singularities \cite{KW},
statistical mechanics of two-dimensional turbulence  \cite{clmp2} and of self-gravitating systems \cite{w} and cosmic strings \cite{pot},
and the theory of hyperelliptic curves \cite{cLin14} and of the Painlev\'e equations \cite{CKLin}. There are by now many results concerning existence \cite{B5,BDeM,BdM2,BdMM,BMal,cama,cl2,cl4,DJLW,dj,EGP,KMdP,linwang},
multiplicity  \cite{BdMM,dem2}, uniqueness \cite{bghjy,bgjm,bjly3,bjl,bl,BLin3,BLT,CCL,GM1,GM3,Lin1,Lin7,suz} and blow up analysis \cite{bcct,bjly2,BM3,bt2,bt,bm,cl1,CLin4,KLin,yy,ls,wz,Za2}.

\

Our goal is to show that bubbling solutions of $\prn$ blowing up at singular points $p_i$ are unique and non-degenerate for $n$ large enough.

\begin{definition}
Let $u_n$ be a sequence of solutions of $\prn$. We say that $u_n$ is a regular $m$-bubbling solution blowing up at the points $q_j\notin\{p_1,\cdots,p_N\}$, $j=1,\cdots,m$, if,
$$
\frac{he^{u_n}}{\int_\Omega h e^{u_n }dx }\rightharpoonup 8\pi\sum\limits_{j=1}^m\delta_{q_j},
$$
weakly in the sense of measures in $\Omega$.

We say that $u_n$ is a singular $m$-bubbling solution blowing up at the points $p_j\in\{p_1,\cdots,p_N\}$, $j=1,\cdots,m$, $m\leq N$ if,
$$
\frac{he^{u_n}}{\int_\Omega h e^{u_n }dx }\rightharpoonup 8\pi\sum\limits_{j=1}^m (1+\alpha_j)\delta_{p_j},
$$
weakly in the sense of measures in $\Omega$.
\end{definition}

\medskip

To state the main result and to compare it with the existing literature we introduce some notation. Let $R(x,y)=\frac{1}{2\pi}\log |x-y|+G(x,y)
$ be the regular part of $G(x,y)$. For what concerns regular bubbling solutions, for $\mathbf{q}=(q_1,\cdots,q_m)\in \om\times\cdots\times \om$, we let
$G_{j}^*(x)=8\pi R(x,q_j)+8\pi\sum^{1,\cdots,m}_{l\neq j}G(x,q_l)$ and
$$
\ell_{\mbox{\footnotesize reg}}(\mathbf{q})=\sum_{j=1}^m[\Delta \log h(q_j)]h(q_j)e^{G_{j}^*(q_j)}.
$$
For $(x_1,\cdots, x_m)\in \om\times \cdots \om$, we also define the $m$-vortex Hamiltonian,
\begin{align}\label{f_qD}
&\mathcal{H}_m(x_1,x_2,\cdots,x_m)=\sum_{j=1}^{m}\big[\log(h(x_j))+4\pi R(x_j,x_j)\big]+4\pi\sum_{l\neq j}^{1,\cdots,m}G(x_l,x_j).
\end{align}
Then, by assuming suitable non-degeneracy conditions the authors in \cite{bjly,bjly2} proved that regular $m$-bubbling solutions are unique and non-degenerate (see also \cite{bjly3} for an analogous result for the Gelfand equation).

\noindent \textbf{Theorem A} (\cite{bjly,bjly2})\textbf{.} \emph{Let $u_{n}^{(1)}$ and $u_{n}^{(2)}$ be  two regular $m$-bubbling solutions of $\prn$,
with $\rho^{(1)}_n=\rho_n=\rho^{(2)}_n$, blowing up at the points $q_j\notin\{p_1,
\cdots,p_N\}$, $j=1,\cdots,m$, where $\mathbf{q}=(q_1,\cdots,q_m)$ is a critical point of $\mathcal{H}_m$.
Assume that, 
\begin{enumerate}
\item $\textrm{det}(D^2\mathcal{H}_m(\mathbf{q}))\neq 0$,
\item $\ell_{\mbox{\footnotesize \emph{reg}}}(\mathbf{q})\neq0$. 
\end{enumerate}
Then there exists $n_0\ge1$ such that $u_{n}^{(1)}=u_{n}^{(2)}$ for all $n\ge n_0$. Moreover, the linearized problem at a $m$-bubbling solution $u_n$
\begin{equation} \label{non-deg}
	\begin{cases}
			\Delta \phi+\rho_n \dfrac{he^{u_n}}{\int_\Omega he^{u_n}dx}\left( \phi-\dfrac{\int_\Omega he^{u_n}\phi\, dx}{\int_\Omega he^{u_n}\,dx} \right)=0 & \mbox{in } \Omega, \\
			\phi=0 & \mbox{on } \partial\Omega,
	\end{cases}
\end{equation}
admits only the trivial solution $\phi\equiv0$ for any $n\ge n_0$.}

\

The above condition (2) can be relaxed by assuming $\ell_{\mbox{\footnotesize reg}}(\mathbf{q})=0$ and $D(\mathbf{q})\neq0$, where $D(\mathbf{q})$ is a geometric quantity. Our aim is to extend the latter result to singular bubbling solutions. Even though the argument works out for more general situations we focus here on singular $1$-bubbling solution blowing up at $p_i$ for some $i\in\{1,\cdots,N\}$, see also Remark \ref{rem}. More precisely, we assume without loss of generality that $\alpha_i\neq\alpha_j$ for $i\neq j$ and we study the case $\rho_n\to8\pi(1+\alpha_i)$ for some fixed $i\in\{1,\cdots,N\}$ and
$$\|u_n\|_{\lf}\to+\infty \quad \mbox{as } n\to+\infty.$$ 
We define 
\begin{equation} \label{l}
\ell(p_i)=\frac{2\pi^2}{(1+\alpha_i)\sin\left(\frac{\pi}{1+\alpha_i}\right)}
\left(\frac{(1+\alpha_i)}{\pi\overline h_1(p_i)}\right)^{\frac{1}{1+\alpha_i}}\Delta\log h_*(p_i),
\end{equation}
where $h(x)=\overline h_1(x)|x-p_i|^{2\alpha_i}$. Moreover, we define the 'desingularized' $1$-vortex Hamiltonian to be
\begin{align} \label{H}
\mathcal{H}_{p_i}(x)=8\pi(1+\alpha_i)\bigr(R(x,p_i)-R(p_i,p_i)\bigr)+\bigr(\log\overline h_1(x)-\log\overline h_1(p_i)\bigr).
\end{align}
Our main results are the following.
\begin{theorem} \label{th.main}
Let $u_{n}^{(1)}$ and $u_{n}^{(2)}$ be  two singular $1$-bubbling solutions of $\prn$,
with $\rho^{(1)}_n=\rho_n=\rho^{(2)}_n$, blowing up at the point $p_i$ for some $i\in\{1,\cdots,N\}$, $\alpha_i\in(0,\infty)\setminus\mathbb{N}$.
Assume that, 
\begin{enumerate}
\item $p_i$ is a critical point of $\mathcal{H}_{p_i}$,
\item $\ell(p_i)\neq0$. 
\end{enumerate}
Then there exists $n_0\ge1$ such that $u_{n}^{(1)}=u_{n}^{(2)}$ for all $n\ge n_0$.
\end{theorem}

\begin{theorem} \label{th.main2}
Let $u_n$ be a singular $1$-bubbling solution of $\prn$, blowing up at the point $p_i$ for some $i\in\{1,\cdots,N\}$, $\alpha_i\in(0,\infty)\setminus\mathbb{N}$. Assume that the conditions (1)-(2) of Theorem \ref{th.main} hold true. Then there exists $n_0\ge1$ such that, for any $n\ge n_0$, \eqref{non-deg} admits only the trivial solution $\phi\equiv0$.
\end{theorem}

\medskip

Observe that we do not need the non-degeneracy of the Hamiltonian as in condition (1) of Theorem A. This is essentially due to the difference of the linearized problem, see \eqref{entirelinear} and the discussion later on. On the other hand, we do need to assume $p_i$ to be a critical point of $\mathcal{H}_{p_i}$. For the regular blow up this is always the case since it is well-known \cite{MaW} that for a regular $m$-bubbling solution blowing up at the points $q_j\notin\{p_1,\cdots,p_N\}$, then $\mathbf{q}=(q_1,\cdots,q_m)$ has to be a critical point of $\mathcal{H}_m$.  

\begin{remark} \label{rem}
The argument yielding Theorems \ref{th.main} and \ref{th.main2} works out for more general situations and can be carried out to prove local uniqueness of singular $m$-bubbling and even for mixed scenarios of singular $m$-bubbling and regular $m'$-bubbling solutions. The decision to focus on singular $1$-bubbling is twofold: on one side the latter case is very subtle since in general the singular blow up point does not need be a critical point of the Hamiltonian $\mathcal{H}_{p_i}$ and furthermore we are not assuming any non-degeneracy of $\mathcal{H}_{p_i}$, and on the other side we want to highlight the differences with respect to the regular case. We postpone the general situation to a future paper. The case $\alpha\in(-1,0)$ will be treated in a separate paper since we first need to derive suitable sharp estimates for bubbling solutions, which are still missing in this case. Finally, the case $\alpha\in\mathbb{N}$ is by now out of reach due to the presence of non-simple (and non-radial) blow up \cite{bt2, KLin}.
\end{remark}

\

To prove Theorem \ref{th.main} we argue by contradiction and we analyze the asymptotic behavior of the (normalized) difference of two distinct solutions for $\prn$,
\begin{equation} \label{xi1}
\xi_n=\frac{u_{n}^{(1)}-u_{n}^{(2)}}{\|u_n^{(1)}-u_n^{(2)}\|_{L^\infty(\Omega)}}.
\end{equation} 
Near the blow up point $p_i$, and after a suitable scaling, $\xi_n$ converges to an entire solution of the linearized problem of the Liouville equation
\begin{equation}
\label{liouville}
\Delta v+ |x|^{2\alpha_i}e^{v}=0\quad\textrm{in}\ \mathbb{R}^2.
\end{equation}
Solutions of \eqref{liouville} with finite mass are completely classified \cite{PT} and for $\alpha_i\in(0,\infty)\setminus\mathbb{N}$ take the form,
\begin{equation}\label{17}
v\left( z\right)=v_{\mu}(z) = \log \frac{8(1+\alpha_i)^2e^{\mu}}{ ( 1+e^{\mu}|z|^{2(1+\alpha_i)})^{2}},  \quad \mu  \in \mathbb{R}.
\end{equation}
The freedom in the choice of $\mu$ is due to the invariance of equation \eqref{liouville} under dilations. The linearized operator $L$ relative to $v_{0}$  is defined by,
\begin{equation}\label{entirelinear}L\phi:=\Delta \phi+\frac{8(1+\alpha_i)^2|z|^{2\alpha_i}}{(1+|z|^{2(1+\alpha_i)})^2}\phi\quad\textrm{in}\ \mathbb{R}^2.
\end{equation}
It follows from \cite[Corollary 2.2]{CLin4} that the $L^{\infty}$-bounded kernel of $L$ has one eigenfunction $Y_0$,
where,
\begin{equation*}
  Y_0(z) = \frac{1-|z|^{2(1+\alpha_i)}}{1+ |z|^{2(1+\alpha_i)}}=\frac{\partial v_{\mu}}{\partial \mu}\Big|_{\mu=0}.
\end{equation*}
The main part of the proof of Theorem \ref{th.main} is to show that, after scaling and for large $n$, $\xi_n$
is orthogonal to $Y_0$. This is done by a delicate analysis of a suitably defined Pohozaev-type identity first introduced in \cite{ly} and then exploited in \cite{bjly, bjly3}.

\medskip

The proof of Theorem \ref{th.main2} follows the same strategy by analyzing the asymptotic behavior of 
$$
\Xi_n = \dfrac{\phi_n-\frac{\int_\Omega he^{u_n}\phi_n\,dx}{\int_\Omega he^{u_n}\,dx}}{\left\|\phi_n-\frac{\int_\Omega he^{u_n}\phi_n\,dx}{\int_\Omega he^{u_n}\,dx}\right\|_{L^{\infty}(\Omega)}}\,,
$$
for a non-trivial solution $\phi_n$ of \eqref{non-deg}, which plays the role of \eqref{xi1}.

\

The paper is organized as follows. In section \ref{sec:prelim} we introduce some preliminary results, in section \ref{sec:norm} we estimate the $L^{\infty}$-norm of the difference of two solutions to $\prn$ and in section \ref{sec:away} we then deduce the first estimates of $\xi_n$, the normalized difference of two solutions, away from the blow up point. In section \ref{sec:poh} we introduce a Pohozaev-type identity to get refined estimates on $\xi_n$ and prove Theorem~\ref{th.main}. Finally, in section \ref{sec:non-deg} we give the sketch of the proof of Theorem \ref{th.main2}.

\

\section{Preliminary estimates about the blow up Phenomenon at the singular point} \label{sec:prelim}
In this section we collect some preliminary results which will be used in the sequel. Let us assume that $i=1$ and set $p=p_1$, $0\neq \alpha=\alpha_1\in(0,+\infty)\setminus\mathbb{N}$. We define
\begin{align} \label{not}
\tilde u_n=u_n-\log\left(\int_{\Omega}he^{u_n}dx\right),\quad
\lambda_n=\max_{\Omega}\tilde u_n,\quad \sigma_n^{2(1+\alpha)}=e^{-\lambda_n},
\end{align}
and
\begin{equation*}
U_n(x)=\lambda_n-2\log(1+\gamma_ne^{\lambda_n}|x-p|^{2+2\alpha}),\quad
\gamma_n=\frac{\rho_n\overline h_1(p)}{8(1+\alpha)^2},
\end{equation*}
where
\begin{align*}
\overline h_1(x)=h_*\exp(-4\pi\overline G_1(x))\quad
\overline G_1(x)=\sum_{i=2}^N\alpha_iG(x,p_i)+R(x,p),
\end{align*}
and $R(x,y)=G(x,y)+\frac{1}{2\pi}\log|x-y|$ is the regular part of the Green function. Therefore, we have
\begin{equation*}
h(x)=\overline h_1(x)|x-p|^{2\alpha},
\end{equation*}
and in any small enough ball centered at $p$ it holds that $\overline h_1>0$. It has been shown in \cite{bcct} (for $\alpha\in(0,+\infty)\setminus\mathbb{N}$) and \cite{bt} (for $\alpha\in(-1,0)$) that
\begin{equation}
\label{2.1}
|\tilde u_n(x)-U_n(x)|\leq C,\quad \forall x\in B_r(p).
\end{equation}
Actually the proofs in \cite{bt2,bcct} show that this estimate holds locally near $p$, but then the global estimate follows by looking at the Definition 1 and the Green representation formula.
\medskip

More recently, it has been proved in \cite{CLin4} that if $\alpha\in(0,+\infty)\setminus\mathbb{N}$, then
\begin{equation}
\label{2.2}
\rho_n-8\pi(1+\alpha)=\ell(p)e^{-\frac{\lambda_n}{1+\alpha}}
+O(e^{-\lambda_n\frac{1+\epsilon_0}{1+\alpha}})~\mathrm{as}~n\to+\infty,
\end{equation}
and
\begin{equation}
\label{2.3}
\rho_{n,1}-8\pi(1+\alpha)=\ell(p)e^{-\frac{\lambda_n}{1+\alpha}}
+O(e^{-\lambda_n\frac{1+\epsilon_0}{1+\alpha}})~\mathrm{as}~n\to+\infty,
\end{equation}
where
\[\rho_{n,1}=\int_{B(p,r_0)}he^{\tilde u_n},\quad
\ell(p)=\frac{2\pi^2}{(1+\alpha)\sin\left(\frac{\pi}{1+\alpha}\right)}
\left(\frac{(1+\alpha)}{\pi\overline h_1(p)}\right)^{\frac{1}{1+\alpha}}\Delta\log h_*(p),\]
and
\[\epsilon_0=2-2(1-\alpha)^+=\begin{cases}
2,\quad &\mathrm{if}~\alpha\geq1,\\
\\
2\alpha, \quad  &\mathrm{if}~\alpha\in(0,1).
\end{cases}\]
Next, we set
\[R_{n,1}=\rho_{n,1}R(x,p),\]
let $r_0>0$ be a small positive number and set
\begin{equation}
\label{2.4}
v_n=\tilde u_n-(R_{n,1}(x)-R_{n,1}(p)),\quad x\in B(p,r_0),
\end{equation}
and as in \cite{Za2}, we denote $\psi_n$ as the solution of
\begin{equation*}
\begin{cases}
\Delta\psi_n=0,\quad &\mathrm{in}~ B_(p,4r_0),\\
\\
\psi_n=v_n-\frac{1}{8\pi r_0}\int_{|x-p|=4r_0}v_nds&\mathrm{on}~\partial B(p,4r_0).
\end{cases}
\end{equation*}
By the Mean value Therorem, we have $\psi(0)=0$. It has been proved in \cite{Za2} that
\begin{equation}
\label{2.5}
v_n-U_n-\psi_n(x)=\sigma_n\psi_{n,1}\left(\frac{x-p}{\sigma_n}\right)
+\sigma_n^2\psi_{n,2}\left(\frac{x-p}{\sigma_n}\right)+O(\sigma_n^2)
~\mathrm{in}~B(p,4r_0),
\end{equation}
where
\begin{equation}
\label{2.6}
\psi_{n,1}(y)=-\frac{2(1+\alpha)a_{n,1}}{\alpha}\frac{y_1}{1+\gamma_n|y|^{2(1+\alpha)}},\quad y=(y_1,y_2)\in\mathbb{R}^2,
\end{equation}
and
\begin{equation}
\label{2.7}
\psi_{n,2}(y)=-a_n\log(2+|y|)+a_{n,0}+O(|y|^{-\epsilon_0}),\quad y=(y_1,y_2)\in\mathbb{R}^2\setminus B(0,R_0),
\end{equation}
for suitable $R_0\geq1$. Here $a_{n,0}$ is a uniformly bounded sequence,
\begin{equation*}
a_n=\frac{\pi}{(1+\alpha)\sin\left(\frac{\pi}{1+\alpha}\right)}
\left(\frac{8(1+\alpha)^2}{\rho_n\overline h_1(p)}\right)^{\frac{1}{1+\alpha}}\Delta\log h_*(p),
\end{equation*}
and, composing with suitable rotations, we can assume that
\begin{equation}
\label{2.8}
(a_{n,1},0)=\nabla\log\left(\overline h_1(x)e^{R_{n,1}(x)+\psi_n(x)}\right)\mid_{x=p}.
\end{equation}
Moreover, it has been shown in \cite[Lemma 3.2]{CLin4} that
\begin{equation}
\label{2.9}
\psi_n(x)=O(\sigma_n^2)\quad, x\in B(p,4r_0).
\end{equation}
Since $\psi_n$ is harmonic, then we also have
\begin{equation}
\label{2.10}
|\nabla\psi_n(x)|=O(\sigma_n^2)\quad, x\in B(p,3r_0).
\end{equation}
We also have, see \cite[Lemma 3.1]{CLin4},
\begin{equation}
\label{2.11}
u_n(x)-\rho_nG(x,p)=O(\sigma_n),\quad x\in\overline\Omega\setminus B(p,r_0).
\end{equation}
Also, we will need the fllowing improved estimate obtained by matching \eqref{2.5} and \eqref{2.11}.

\begin{lemma}
\label{le1.1} It holds,
\begin{equation}
\label{2.12}
\lambda_n-\log\left(\int_{\Omega}he^{u_n}\right)+2\log\gamma_n+8\pi(1+\alpha)R(p,p)=O(\sigma_n).
\end{equation}
\end{lemma}
\begin{proof}
Putting $c_n=\log\left(\int_{\Omega}he^{u_n}\right)$ and picking any $|x-p|=2r_0$ in \eqref{2.5} and \eqref{2.11}, we conclude that
\begin{equation*}
\rho_nG(x,p)-c_n-(R_{n,1}(x)-R_{n,1}(p))-U_n(x)=O(\sigma_n).
\end{equation*}
Clearly we have
\begin{equation*}
U_n(x)=-\lambda_n-4(1+\alpha)\log|x-p|-2\log(\gamma_n)+O(\sigma_{n}^{2(1+\alpha)}),
\end{equation*}
and we find that
\begin{align*}
&-\frac{\rho_n}{2\pi}\log|x-p|+\rho_nR(x,p)-c_n-\rho_{n,1}R(x,p)+\rho_{n,1}R(p,p)\\
&\quad +\lambda_n+4(1+\alpha)\log|x-p|+2\log(\gamma_n)=O(\sigma_n),
\end{align*}
and then the desired conclusion easily follows from \eqref{2.2} and \eqref{2.3}.
\end{proof}
Finally, similar arguments used in the estimate \eqref{2.11}, yield
\begin{equation}
\label{2.13}
\nabla (\tilde u_n-\rho_nG(x,p))=O(\sigma_n),\quad  x\in \overline\Omega\setminus B(p,r_0).
\end{equation}

\

\section{Estimate of the $L^\infty$-norm} \label{sec:norm}
The proof of Theorem \ref{th.main} is obtained by contradiction and we assume that two distinct solutions $u_n^{(i)},~i=1,2$, exist for $\prn$, whence in particular with the same $\rho_n$, which satisfy
\begin{equation*}
\rho_n\to 8\pi(1+\alpha)\quad\mathrm{as}\quad n\to+\infty,
\end{equation*}
where $\alpha=\alpha_1$. We also assume without loss of generality that
\[p_1=0\in\Omega.\]
Then we define
\[\tilde u_n^{(i)}=u_n^{(i)}(x)-\log\left(\int_\Omega he^{u_n^{(i)}}\right),
\quad  \lambda_n^{(i)}=\max_{\overline\Omega}\tilde u_n^{(i)},\]
and in particular $v_n^{(i)}$ defined as in \eqref{2.4}. Also we set
\begin{equation*}
U_n^{(i)}(x)=\lambda_n^{(i)}-2\log(1+\gamma_ne^{\lambda_n^{(i)}}|x-p|^{2(1+\alpha)}),\quad i=1,2,\quad \gamma_n=\frac{\rho_n\overline h_1(p)}{8(1+\alpha)^2}.
\end{equation*}
There is no loss of generality in assuming that
$$\lambda_n^{(1)}\leq \lambda_{n}^{(2)}.$$
To simplify the notation, we set
$$
\sigma_n^{2(1+\alpha)}=e^{-\lambda_n^{(1)}}.
$$
Then we have

\begin{lemma}
\label{le3.1}
\begin{itemize}
\item [(i)] $|\lambda_n^{(1)}-\lambda_{n}^{(2)}|=O(\sum_{i=1}^2e^{-\frac{\epsilon_0}{\alpha+1}\lambda_n^{(i)}})
    =O(e^{-\frac{\epsilon_0}{\alpha+1}\lambda_n^{(1)}})=O(\sigma_n^{2\epsilon_0}).$

\vspace{0.2cm}

\item [(ii)] $\|\tilde u_n^{(1)}-\tilde u_n^{(2)}\|_{L^\infty(B(0,r_0))}\leq |\lambda_n^{(2)}-\lambda_{n}^{(1)}|+O(\lambda_n^{(1)}\sigma_n^2)$.

\vspace{0.2cm}

\item [(iii)] $\|\tilde u_n^{(1)}-\tilde u_n^{(2)}\|_{L^\infty(\Omega\setminus B(0,r_0))}\leq O(\sigma_n)$.
\end{itemize}
\end{lemma}

\begin{proof}
(i) In view of \eqref{2.2}, we find that
\[\ell(p)e^{-\frac{\lambda_n^{(1)}}{1+\alpha}}+O\left(e^{-\frac{1+\epsilon_0}{1+\alpha}\lambda_n^{(1)}}\right)
=\ell(p)e^{-\frac{\lambda_n^{(2)}}{1+\alpha}}+O\left(e^{-\frac{1+\epsilon_0}{1+\alpha}\lambda_n^{(2)}}\right),\]
which immediately implies, since $\ell(p)\neq0$,
\[\lambda_n^{(1)}-\lambda_n^{(2)}=O(e^{-\frac{\epsilon_0}{1+\alpha}\lambda_n^{(1)}})
+e^{-\frac{\epsilon_0}{1+\alpha}\lambda_n^{(2)}},\]
as claimed.

\medskip

(ii) By using $\lambda_n^{(1)}\leq \lambda_n^{(2)}$, it is not difficult to see that
\begin{align*}
U_n^{(2)}-U_n^{(1)}=~&(\lambda_n^{(2)}-\lambda_n^{(1)})
\left(\frac{1-e^{\lambda_n^{(1)}}\gamma_n|x-p|^{2(1+\alpha)}}{1+e^{\lambda_n^{(1)}}\gamma_n|x-p|^{2(1+\alpha)}}\right)
+O((\lambda_n^{(2)}-\lambda_n^{(1)})^2)\\
\leq~& |\lambda_n^{(2)}-\lambda_n^{(1)}|+O((\lambda_n^{(2)}-\lambda_n^{(1)})^2),
\end{align*}
uniformly in $B(0,r)$ for any $r>0$. Also, in view of \eqref{2.9}, and since the $\psi_n^{(i)}$'s are harmonic, we find that
\[\psi_n^{(2)}(x)-\psi_n^{(1)}(x)=O(\sigma_n^2),\quad |\nabla(\psi_n^{(2)}(x)-\psi_n^{(1)}(x))|=O(\sigma_n^2),\]
uniformly in $B(0,3r_0)$, we use this gradient estimate to evaluate the difference,
$$|a_{n,2}-a_{n,1}|=|\nabla(\psi_n^{(2)}(0)-\psi_n^{(1)}(0))|=O(\sigma_n^2),$$
which implies that
\begin{align*}
|\psi_{n,2}^{(2)}(x)-\psi_{n,1}^{(1)}(x)|\leq&~\frac{2(1+\alpha)}{\alpha}|a_{n,1}^{(2)}-a_{n,1}^{(1)}||x_1|
+O(\lambda_n^{(2)}-\lambda_n^{(1)})\\
=&~O(\sigma_n^2)+O(\lambda_n^{(2)}-\lambda_n^{(1)}),
\end{align*}
uniformly in $B(0,r_0)$. Also it is easy to see that
\begin{equation*}
\left|\psi_{n,2}^{(2)}-\psi_{n,2}^{(1)}\right|=O(\lambda_n^{(1)}\sigma_n^2).
\end{equation*}
Therefore, in view of \eqref{2.5} and Lemma \ref{le3.1}, we finally conclude that,
\begin{equation*}
\|\tilde u_n^{(1)}(x)-\tilde u_n^{(2)}(x)\|_{L^\infty(B(0,r_0))}\leq |\lambda_n^{(2)}-\lambda_n^{(1)}|+O(\lambda_{n}^{(1)}\sigma_n^2),
\end{equation*}
which is (ii).

\medskip

(iii) Next we obtain the estimate in $\Omega\setminus B(0,r_0)$, by using the Green's representation formula,
\begin{equation*}
\begin{aligned}
\tilde u_n^{(1)}(x)-\tilde u_n^{(2)}(x)=~&\rho_n\int_\Omega G(y,x)h(y)(e^{\tilde u_n^{(1)}(y)}-e^{\tilde u_n^{(2)}(y)})dy\\
=~&\rho_n\int_{B(0,r_0)}(G(y,x)-G(0,x))h(y)(e^{\tilde u_n^{(1)}(y)}-e^{\tilde u_n^{(2)}(y)})dy\\
&+G(0,x)\int_{B(0,r_0)}\rho_nh(y)(e^{\tilde u_n^{(1)}(y)}-e^{\tilde u_n^{(2)}(y)})dy\\
&+\rho_n\int_{\Omega\setminus B(0,r_0)}G(y,x)h(y)(e^{\tilde u_n^{(1)}(y)}-e^{\tilde u_n^{(2)}(y)})dy.
\end{aligned}
\end{equation*}
In view of \eqref{2.1} and since $\rho_n$ is the same for the two solutions, then we have
\begin{align*}
\rho_{n,1}^{(1)}-\rho_{n,1}^{(2)}=~&\rho_n\int_{B(0,r_0)}h(y)e^{\tilde u_n^{(1)}(y)}
-\rho_n\int_{B(0,r_0)}h(y)e^{\tilde u_n^{(2)}(y)}\\
=~&\rho_n\int_{\Omega\setminus B(0,r_0)}h(y)\left(e^{\tilde u_n^{(2)}(y)}-e^{\tilde u_n^{(1)}(y)}\right)dy=O(e^{-\lambda_n}).
\end{align*}
Then, by using \eqref{2.1} once more, for $x\in\Omega\setminus B(0,r_0)$ we have,
\begin{align*}
&\tilde{u}_n^{(1)}(x)-\tilde u_n^{(2)}(x)\\
&=\rho_n\int_{B(0,r_0)}(G(y,x)-G(0,x))h(y)\left(e^{\tilde u_n^{(2)}(y)}-e^{\tilde u_n^{(1)}(y)}\right)dy\\
&\quad +G(0,x)(\rho_{n,1}^{(1)}-\rho_{n,2}^{(2)})+O(e^{-\lambda_n})\\
&=\rho_n\int_{B(0,r_0)}(G(y,x)-G(0,x))h(y)\left(e^{\tilde u_n^{(2)}(y)}-e^{\tilde u_n^{(1)}(y)}\right)dy+O(e^{-\lambda_n})\\
&=\int_{B(0,r_0)}O(1)\left(\sum_{i=1,2}\frac{|y|^{2\alpha+1}e^{\lambda_n^{(i)}}}
{(1+\gamma_ne^{\lambda_n^{(i)}}|y|^{2+2\alpha})^2}\right)dy+O(e^{-\lambda_n})\\
&=O(\sigma_n),
\end{align*}
uniformly in $x\in\Omega\setminus B(0,r_0)$. Therefore we conclude that
\begin{equation*}
\|\tilde u_n^{(1)}(x)-\tilde u_n^{(2)}(x)\|_{L^\infty(\Omega\setminus B(0,r_0))}\leq O(\sigma_n),
\end{equation*}
as claimed.
\end{proof}

\

\section{Estimate of the difference away from the blow up point} \label{sec:away}
Let
\begin{equation} \label{xi}
\xi_n=\frac{\tilde u_n^{(1)}-\tilde{u}_n^{(2)}}{\|\tilde u_n^{(1)}-\tilde{u}_n^{(2)}\|_{L^\infty(\Omega)}}.
\end{equation}
Clearly $\xi_n$ satisfies
\begin{equation}
\label{4.1}
\begin{cases}
\Delta\xi_n+\rho_nh(x)c_n(x)\xi_n(x)=0\quad &\mathrm{in}~\Omega,\\
\\
\xi_n=-d_n\quad &\mathrm{on}~\partial\Omega,
\end{cases}
\end{equation}
for some constant $d_n$ satisfying $|d_n|\leq 1$ and
\[c_n(x)=\frac{e^{\tilde u_n^{(1)}}-e^{\tilde{u}_n^{(2)}}}{\tilde u_n^{(1)}-\tilde{u}_n^{(2)}}.\]
To simplify the notations, we set
\[\lambda_n=\lambda_n^{(1)}\quad\mathrm{and}\quad \sigma_n^{2+2\alpha}=e^{-\lambda_n}.\]
Then by defining
$$\hat\xi_n(z)=\xi_n(\sigma_n z),\quad |z|<4\sigma_n^{-1}r_0,$$
we prove the following

\begin{lemma}
\label{le4.1}
There exists a constant $b_0\in\mathbb{R}$, such that $\hat\xi_n(z)\to b_0\hat\xi_0(z)$ in $C_{\mathrm{loc}}^0(\mathbb{R}^2)$, where
\[\hat\xi_0(z)=\frac{1-\gamma |z|^{2+2\alpha}}{1+\gamma |z|^{2+2\alpha}},\quad z\in\mathbb{R}^2,\]
where $\gamma=\frac{\pi\overline h_1(0)}{1+\alpha}.$
\end{lemma}

\begin{proof}
By Lemma \ref{le3.1}, we see that
\begin{align*}
c_n(x)=~&e^{\tilde u_n^{(1)}(x)}\left(1+O(\|\tilde u_n^{(1)}-\tilde{u}_n^{(2)}\|_{L^\infty(\Omega)})\right)\\
=~&e^{\tilde u_n^{(1)}(x)}(1+O(|\lambda_n^{(2)}-\lambda_n^{(1)}|+\sigma_n)),
\end{align*}
and then by \eqref{2.5}, \eqref{2.9} and \eqref{2.10}
\begin{equation*}
e^{-\lambda_n}c_n(\sigma_nz)=\dfrac{e^{C\sigma_n(1+O(|\lambda_n^{(2)}-\lambda_n^{(1)}|+\sigma_n))}}
{(1+\gamma_n|z|^{2+2\alpha})^2}
\to\frac{1}{(1+\gamma|z|^{2+2\alpha})^2}
\quad\mathrm{in}~C_{\mathrm{loc}}^2(\mathbb{R}^2),
\end{equation*}
where $\gamma=\frac{\pi\overline h_1(0)}{1+\alpha}.$

We define
\begin{equation*}
\Omega_{\sigma_n}=\left\{z\in\mathbb{R}^2\mid \sigma_nz\in\Omega\right\}.
\end{equation*}
By using \eqref{4.1}, we have
\begin{equation*}
\begin{cases}
\Delta\hat\xi_n+\rho_n\overline h_1(\sigma_nz)|z|^{2\alpha}\sigma_n^{2+2\alpha}\quad &\mathrm{in}\quad \Omega_{\sigma_n},\\
\\
\hat\xi_n(z)=-d_n &\mathrm{on}\quad \partial\Omega_{\sigma_n},
\end{cases}
\end{equation*}
and since $|\hat\xi_n|\leq 1$, then we conclude that $\hat\xi_n\to\hat\xi$ in $C_{\mathrm{loc}}^0(\mathbb{R}^2)$, where $\hat\xi$ is a solution of
\begin{equation*}
\Delta\hat\xi+\dfrac{8\gamma (1+\alpha)^2|z|^{2\alpha}}{(1+\gamma|z|^{2(1+\alpha)})^2}\hat\xi=0~\mathrm{in}~\mathbb{R}^2 \quad \mathrm{and}\quad |\hat\xi(z)|\leq1~\mathrm{in}~\mathbb{R}^2.
\end{equation*}
It follows from \cite[Corollary 2.2]{CLin4} that $\hat\xi(z)=b_0\xi_0(z)$, for some constant $b_0$, as claimed.
\end{proof}

Next, we have
\begin{lemma}
\label{le4.2}
For any $r_0$ small enough we have
$$\xi_n(x)=-b_0+o(1),\quad x\in\Omega\setminus B(0,r_0),$$
where $b_0$ is defined by Lemma \ref{le4.1}.
\end{lemma}

\begin{proof}
It follows from \eqref{2.1} that
$$c_n(x)\to0\quad\mathrm{in}\quad C_{\mathrm{loc}}^0(\overline\Omega\setminus\{0\}).$$	
Since $\|\xi_n\|_{L^\infty(\Omega)}\leq 1$, then \eqref{4.1} implies that
$$\xi_n\to\xi_0\quad\mathrm{in}\quad C_{\mathrm{loc}}^0(\overline\Omega\setminus\{0\}),$$
where
$$\Delta\xi_0=0~\mathrm{in}~\Omega\setminus\{0\}\quad\mathrm{and}\quad\|\xi_0\|_{L^\infty(\Omega)}\leq1.$$
As a consequence, $\xi_0$ is smooth in $\Omega\setminus\{0\}$ and in particular
$$\Delta\xi_0=0\quad\mathrm{in}\quad\Omega.$$
Therefore $\xi_0=-b$ in $\Omega$ for some constant $b$ and
\begin{equation}
\label{4.2}
\xi_n\to-b\quad\mathrm{in}\quad C_{\mathrm{loc}}^0(\overline\Omega\setminus\{0\}).
\end{equation}
In particular $-\xi_n(x)=d_n\to b$ for $x\in \partial\Omega.$ Let
$\phi_n=\frac{1-\gamma_n e^{\lambda_n}|x|^{2+2\alpha}}{1+\gamma_n e^{\lambda_n}|x|^{2+2\alpha}}$ and let us fix $d\in(0,r_0)$. Then, by using \eqref{2.5}, \eqref{2.9} and \eqref{2.10}, we find that
\begin{equation*}
\begin{aligned}
&\int_{\partial B(0,d)}\left(\phi_n\frac{\partial\xi_n}{\partial\nu}-\xi_n\frac{\partial\phi_n}{\partial\nu}\right)d\sigma
=\int_{B(0,d)}(\phi_n\Delta\xi_n-\xi_n\Delta\phi_n)dx\\
&=\int_{B(0,d)}\left\{-\rho_n\xi_n\phi_n\overline h_1(x)|x|^{2\alpha}\left(\frac{e^{\tilde u_n^{(1)}}-e^{\tilde{u}_n^{(2)}}}{\tilde u_n^{(1)}-\tilde u_n^{(2)}}\right)+8(1+\alpha)^2\gamma_n\xi_n\phi_n|x|^{2\alpha}e^{U_n}\right\}dx\\
&=\int_{B(0,d)}\rho_n\xi_n\phi_n\left\{-\overline h_1(x)|x|^{2\alpha}e^{\tilde u_n^{(1)}}
(1+O(|\tilde u_n^{(1)}-\tilde u_n^{(2)}|))+\overline h_1(0)|x|^{2\alpha}e^{U_n}\right\}dx\\
&=\int_{B(0,d)}\rho_n\xi_n\phi_n|x|^{2\alpha}e^{U_n}\left\{-\overline h_1(x)e^{O(\sigma_n)}(1+O(|\tilde u_n^{(1)}-\tilde u_n^{(2)}|))+\overline h_1(0)\right\}dx.
\end{aligned}
\end{equation*}
Therefore, by the scaling $x=\sigma_nz$, we see that,
\begin{align*}
&\int_{\partial B(0,d)}\left(\phi_n\frac{\partial\xi_n}{\partial\nu}
-\xi_n\frac{\partial\phi_n}{\partial\nu}\right)d\sigma\\
&=\int_{B(0,d/\sigma_n)}\rho_n\hat\xi_n(z)\hat\phi_n(z)|z|^{2\alpha}
\frac{O(1)(\sigma_n|z|+|\hat u_n^{(1)}-\hat u_n^{(2)}|+\sigma_n)}{(1+\gamma_n|z|^{2+2\alpha})^2}dz.
\end{align*}
In view of Lemma \ref{le3.1} we obtain
\begin{equation}
\label{4.3}
\int_{\partial B(0,d)}\left(\phi_n\frac{\partial\xi_n}{\partial\nu}
-\xi_n\frac{\partial\phi_n}{\partial\nu}\right)d\sigma=O(\sigma_n+\sigma_n^{2\epsilon_0}).
\end{equation}
Let $\zeta_n=\int_0^{2\pi}\xi_n(r,\theta)d\theta$, where $r=|x|$. Then, for any fixed $R>0$, \eqref{4.3} yields
\begin{equation*}
(\zeta_n)'(r)\phi_n(r)-\zeta_n(r)\phi_n'(r)=\frac{O(\sigma_n+\sigma_n^{2\epsilon_0})}{r},\quad \forall r\in(R\sigma_n,r_0].
\end{equation*}
Also for any $R>0$ large enough, and for any $r\in(R\sigma_n,r_0]$, we also obtain that
\begin{equation*}
\phi_n(r)=-1+O\left(\frac{\sigma_n^{2+2\alpha}}{r^{2+2\alpha}}\right),\quad
\phi_n'(r)=O\left(\frac{\sigma_n^{2+2\alpha}}{r^{3+2\alpha}}\right),
\end{equation*}
and so we conclude that
\begin{equation}
\label{4.4}
\zeta_n'(r)=\frac{O(\sigma_n+\sigma_n^{2\epsilon_0})}{r}
+O\left(\frac{\sigma_n^{2+2\alpha}}{r^{3+2\alpha}}\right),
\quad \forall r\in(R\sigma_n,r_0].
\end{equation}
Integrating \eqref{4.4} we obtain that
\begin{equation}
\label{4.5}
\zeta_n(r)=\zeta_n(R\sigma_n)+o(1)+O(R^{-(2+2\alpha)}),\quad \forall r\in(R\sigma_n,r_0].
\end{equation}
In view of Lemma \ref{le4.1}, we also have
$$\zeta_n(R\sigma_n)=-2\pi b_0+o_R(1)+o_n(1),$$
where $\lim_{R\to+\infty}o_R(1)=0$ and $\lim_{n\to+\infty}o_n(1)=0$. Then by \eqref{4.5} we have
\begin{equation}
\label{4.6}
\zeta_n(r)=-2\pi b_0+o_R(1)+o_n(1)(1+O(R)),\quad \forall r\in(R\sigma_n,r_0].
\end{equation}
In view of \eqref{4.2}, we see that
$$\zeta_n=-2\pi b+o_n(1)~\mathrm{in}~C_{\mathrm{loc}}(\Omega\setminus\{0\}),$$
which implies that $b=b_0$. Hence, we finish the proof.
\end{proof}

Next, we need a refined estimate about $\xi_n$ which will be needed in next section.

\begin{lemma}
\label{le4.3}
\begin{equation}
\label{4.7}
\xi_n(x)=-d_n+A_nG(0,x)+o(\sigma_n)\quad\mathrm{in}\quad C^1(\Omega\setminus B(0,2r_0)),
\end{equation}
where
\begin{equation*}
A_n=\int_\Omega f_n^*(x)\quad \mathrm{and}\quad f_n^*(x)=\rho_nc_n(x)h(x)\xi_n(x).
\end{equation*}
Moreover, there is a constant $C>0$, which does not depend on $R>0$, such that
\begin{equation}
\label{4.8}
|\xi_n(x)+d_n-A_nG(0,x)|\leq C\sigma_n\left(\frac{1_{B(0,2r_0)}(x)}{|x|}+1_{\Omega\setminus B(0,2r_0)}(x)\right),~ x\in \Omega\setminus B(0,R\sigma_n).
\end{equation}
\end{lemma}

\begin{proof}
By the Green representation formula we find that,
\begin{equation}
\label{4.9}
\begin{aligned}
\xi_n(x)=&-d_n+\int_{\Omega}G(y,x)f_n^*(y)dy\\
=&-d_n+A_nG(0,x)+\int_{\Omega}(G(y,x)-G(0,x))f_n^*(y)dy,
\end{aligned}
\end{equation}
while, by Lemma \ref{le3.1}, we also find that
\begin{equation}
\label{4.10}
c_n(x)\xi_n(x)=\frac{e^{\tilde u_n^{(1)}}-e^{\tilde u_n^{(2)}}}{\|\tilde u_n^{(1)}-\tilde u_n^{(2)}\|_{L^\infty(\Omega)}}=e^{\tilde u_n^{(1)}}\xi_n(x)(1+O(\lambda_n^{(2)}-\lambda_{n}^{(1)}+\sigma_n)).
\end{equation}
Thus, for $x\in\Omega\setminus B(0,2r_0)$, we see from \eqref{2.1}, \eqref{2.5} that
\begin{equation}
\label{4.11}
\begin{aligned}
&\int_{\Omega}(G(y,x)-G(0,x))f_n^*(y)dy=\int_{B(0,r_0)}(G(y,x)-G(0,x))f_n^*(y)dy+O(e^{-\lambda_n})\\
&=\int_{B(0,r_0)}\left\langle\partial_yG(y,x)\mid_{y=0},y\right\rangle f_n^*(y)dy
+O(1)\left(\frac{|y|^{2+2\alpha}e^{\lambda_n}}{(1+e^{\lambda_n}|y|^{2+2\alpha})^2}dy\right)\\
&=\int_{B(0,r_0)}\left\langle\partial_yG(y,x)\mid_{y=0},y\right\rangle f_n^*(y)dy+O(\sigma_n^2).
\end{aligned}
\end{equation}
By using \eqref{2.1}, \eqref{2.5} and Lemma \ref{le3.1}, after scaling we see that for $x\in\Omega\setminus B(0,2r_0)$, it holds
\begin{align*}
&\int_{B(0,r_0)}\left\langle\partial_yG(y,x)\mid_{y=0},y\right\rangle f_n^*(y)dy\\
&=\sigma_n^{3+2\alpha}\int_{B(0,r_0/\sigma_n)}\left\langle\partial_yG(y,x)\mid_{y=0},z\right\rangle\rho_n\overline h_1(\sigma_nz)|z|^{2\alpha}e^{\hat U_n+\hat R_{n,1}(z)-\hat R_{n,1}(0)}\hat\xi_n(z)dz\\
&\quad+O(\sigma_n^{1+2\epsilon_0}+\sigma_n^2)\\
&=\sigma_n\int_{B(0,r_0/\sigma_n)}\dfrac{\left\langle\partial_yG(y,x)\mid_{y=0},z\right\rangle\rho_n\overline h_1(0)|z|^{2\alpha}\hat \xi_n(z)}{(1+\gamma_n|z|^{2+2\alpha})^2}dz
+O(\sigma_n^{1+2\epsilon_0}+\sigma_n^2).
\end{align*}
Therefore, in view of Lemma \ref{le4.1}, for $x\in\Omega\setminus B(0,2r_0)$ we find that,
\begin{equation}
\label{4.12}
\begin{aligned}
&\int_{B(0,r_0)}\left\langle\partial_yG(y,x)\mid_{y=0},y\right\rangle f_n^*(y)dy\\
&=\sigma_n\sum_{h=1}^2\partial_{y_h}G(y,x)\mid_{y=0}\rho_n\overline h_1(0)b_0
\int_{B(0,r_0/\sigma_n)}\frac{z_h|z|^{2\alpha}\hat\xi_0(z)}{(1+\gamma_n|z|^{2+2\alpha})^2}dz+o(\sigma_n)\\
&=\sigma_n\sum_{h=1}^2\partial_{y_h}G(y,x)\mid_{y=0}\rho_n\overline h_1(0)b_0
\int_{\mathbb{R}^2}\frac{z_h|z|^{2\alpha}\hat\xi_0(z)}{(1+\gamma_n|z|^{2+2\alpha})^2}dz+o(\sigma_n).
\end{aligned}
\end{equation}
From \eqref{4.9}-\eqref{4.12}, we see that the estimate \eqref{4.7} holds in $C^0(\Omega\setminus B(0,r_0))$. The proof of the fact that \eqref{4.7} holds in $C^1(\Omega\setminus B(0,r_0))$ is similar and we skip it here to avoid repetitions.

From \eqref{4.10}, \eqref{2.5} and suitable scaling, we see that there exists $C>0$, which is independent of $R>0$ such that for $x\in B(0,2r_0)\setminus B(0,\sigma_nR)$, it holds that
\begin{equation}
\label{4.13}
\begin{aligned}
&|\xi_n(x)+d_n-A_nG(0,x)|\leq|\int_{B(0,3r_0)}(G(y,x)-G(0,x))f_n^*(y)dy|+O(e^{-\lambda_n})\\
&\leq|\frac{1}{2\pi}\int_{B(0,3r_0)}\log\frac{|x|}{|x-y|}f_n^*(y)dy|+O\left(\int_{B(0,3r_0)}
\frac{e^{\lambda_n}|y|^{1+2\alpha}}{(1+e^{\lambda_n}|y|^{2+2\alpha})^2}dy\right)+O(e^{-\lambda_n})\\
&\leq O(1)\left(\int_{B(0,3r_0/\sigma_n)}\frac{\left|\log|x|-\log|x-\sigma_nz|\right|
|z|^{2\alpha}}{(1+|z|^{2+2\alpha})^2}dz\right)
+O(\sigma_n)\\
&\leq O(1)\left(\int_{\frac{|x/\sigma_n|}{2}\leq|z|\leq2|x/\sigma_n|}
\frac{\left|\log|x|-\log|x-\sigma_nz|\right||z|^{2\alpha}}{(1+|z|^{2+2\alpha})^2}dz\right)\\
&\quad+O(1)\left(\int_{B(0,3r_0/\sigma_n)}\frac{\sigma_n|z|^{2\alpha+1}}{|x|(1+|z|^{2+2\alpha})^2}dz\right)
+O(\sigma_n)\\
&\leq O(1)\left(\frac{\sigma_{n}}{|x|}\right)
+O(1)(\log|z||z|^{-2}\mid_{|z|=|x|/\sigma_n})+O(\sigma_n)\leq C\left(\frac{\sigma_n}{|x|}\right).
\end{aligned}
\end{equation}
By \eqref{4.9}, \eqref{4.10} and \eqref{2.5}, we also see that for $x\in \Omega\setminus B(0,2r_0)$, it holds that
\begin{equation}
\label{4.14}
|\xi_n(x)+d_n-A_nG(x,0)|=O\left(\int_{B(0,r_0)}\frac{e^{\lambda_n}|y|^{1+2\alpha}}
{(1+e^{\lambda_n}|y|^{2+2\alpha})^2}dy\right)+O(e^{-\lambda_n})=O(\sigma_n).
\end{equation}
By \eqref{4.13} and \eqref{4.14} we obtain \eqref{4.8}, which concludes the proof of Lemma \ref{le4.3}.
\end{proof}

\

\section{Estimates via Pohozaev identities} \label{sec:poh}
From now on, for a given function $f(y,x)$, we shall use $\partial$ and $D$ to denote the partial derivatives with respect to $y$ and $x$ respectively. With a small abuse of notation, for a function $f(x)$ we will use both $\nabla$ and $D$ to denote its gradient.

We define
\begin{equation}
\label{5.1}
\varphi_n(y)=\rho_{n}(R(y,0)-R(0,0)),
\end{equation}
and
\begin{equation}
\label{5.2}
v_n^{(i)}=\tilde u_n^{(i)}-\varphi_n(y),\quad  i=1,2.
\end{equation}
Recall the definition of $\xi_n$ which satisfies \eqref{4.1}. Our aim is to show that the projection of $\xi_n$ on the radial part kernel is zero, i.e., $b_0=0$. We shall accomplish it by exploiting the following Pohozaev identity to derive a more accurate estimate on $\xi_n.$

\begin{lemma}\emph{(\cite{ly})}
\label{le5.1}
For any fixed $r\in(0,r_0)$, it holds
\begin{equation}
\label{5.3}
\begin{aligned}
&\frac12\int_{\partial B(0,r)}r\left\langle Dv_n^{(1)}+Dv_n^{(2)},D\xi_n\right\rangle d\sigma
-\int_{\partial B(0,r)}r\left\langle\nu,D(v_n^{(1)}+v_n^{(2)})\right\rangle\left\langle\nu,D\xi_n\right\rangle d\sigma\\
&=\int_{\partial B(0,r)}\frac{r\rho_nh(x)}{\|v_n^{(1)}-v_n^{(2)}\|_{L^\infty(\Omega)}}
(e^{v_n^{(1)}+\varphi_n}-e^{v_n^{(2)}+\varphi_n})d\sigma\\
&\quad-\int_{B(0,r)}\frac{\rho_nh(x)(e^{v_n^{(1)}+\varphi_n}-e^{v_n^{(2)}+\varphi_n})}
{\|v_n^{(1)}-v_n^{(2)}\|_{L^\infty(\Omega)}}
\left(2+2\alpha+\left\langle D(\log\overline h_1(x)+\varphi_n(x)),x\right\rangle\right)dx.
\end{aligned}
\end{equation}
\end{lemma}

\begin{proof}
See \cite{bjly} for a proof of this identity.
\end{proof}

Let
\begin{align} \label{P}
\begin{split}
	\Phi(y,0)= &-8\pi(1+\alpha)\log|y|+8\pi(1+\alpha)(R(y,0)-R(0,0)) \\
	&+\log(\overline h_1(y))-\log(\overline h_1(0)).
\end{split}
\end{align}
Recall the definition of $A_n$ given in Lemma \ref{le4.3}. Then we have

\begin{lemma}
\label{le5.2}
\begin{align*}
\mathrm{L.H.S.~of}~\eqref{5.3}=~&-4(1+\alpha)A_n-\frac{(8(1+\alpha)^2)^3b_0e^{-\lambda_n}}{2\rho_n\overline h_1(0)}\int_{\Omega\setminus B(0,r)}|y|^{2\alpha}e^{\Phi(y,0)}\\
&+o(\sigma_n^2)+O(\sigma_n|A_n|)+O(r^{-3}\sigma_n^3).
\end{align*}
\end{lemma}

\begin{proof}
Let
\begin{equation}
\label{5.4}
G_n(x)=\rho_nG(x,0),
\end{equation}	
so that
\begin{equation}
\label{5.5}
\nabla (G_n(x)-\varphi_n)(x)=-\frac{\rho_n}{2\pi}\frac{x}{|x|^2}.
\end{equation}	
In view of \eqref{2.13}, we have
\begin{align*}
\nabla v_n^{(i)}(x)=~&\nabla(\tilde{u}_n^{(i)}-G_n(x))+\nabla(G_n(x)-\varphi_n(x))\\
=~&\nabla(G_n(x)-\varphi_n(x))+O(\sigma_n),\quad  x\in\overline\Omega\setminus B(0,r_0),
\end{align*}	
for any fixed small $r_0>0.$ As a consequence, for fixed $r>r_0$, we find that
\begin{equation}
\label{5.6}
\begin{aligned}
\mathrm{L.H.S.~of}~\eqref{5.3}=~&\int_{\partial B(0,r)}r\left\langle D(G_n-\varphi_n),D\xi_n\right\rangle d\sigma-2\int_{\partial B(0,r)}r\left\langle\nu,D(G_n-\varphi_n)\right\rangle\left\langle\nu,D\xi_n\right\rangle d\sigma\\
&+O(\sigma_n\|D\xi_n\|_{L^\infty(\partial B(0,r))})\\
=~&\int_{\partial B(0,r)}\frac{\rho_n}{2\pi}\left\langle D\xi_n,\nu\right\rangle d\sigma
+O(\sigma_n\|D\xi_n\|_{L^\infty(\partial B(0,r))})\\
=~&\int_{\partial B(0,r)}4(1+\alpha)\left\langle D\xi_n,\nu\right\rangle d\sigma
+O(\sigma_n\|D\xi_n\|_{L^\infty(\partial B(0,r))}),
\end{aligned}
\end{equation}
where we used \eqref{2.2}. Therefore, as a consequence of Lemma \ref{le4.3}, we conclude that
\begin{equation}
\label{5.7}
\mathrm{L.H.S.~of}~\eqref{5.3}=4(1+\alpha)\int_{\partial B(0,r)}\left\langle D\xi_n,\nu\right\rangle d\sigma+O(\sigma_n|A_n|)+o(\sigma_n^2).
\end{equation}
In this particular case, we have $A_n=0.$

To estimate the right hand side of \eqref{5.7}, we need a refined estimate about $\xi_n$ on $\partial B(0,r)$. So, by the Green representation formula with $x\in\partial B(0,r)$, we find that
\begin{equation}
\label{5.8}
\begin{aligned}
\xi_n(x)=~&-d_n+\int_{\Omega}G(y,x)f_n^*(y)dy\\
=~&-d_n+A_nG(0,x)+\sum_{h=1}^2B_{n,h}\partial_{y_h}G(y,x)\mid_{y=0}
+\frac12\sum_{h,k=1}^2C_{n,h,k}\partial_{y_hy_k}^2G(y,x)\mid_{y=0}\\
&+\int_{\Omega}\Psi_n(y,x)f_n^*(y),
\end{aligned}
\end{equation}
where
\[A_n=\int_{\Omega}f_n^*(y)dy,\quad B_{n,h}=\int_{B(0,r)}y_hf^*_n(y)dy,\quad C_{n,h,k}=\int_{B(0,r)}y_hy_kf_n^*(y),\]
and
\begin{align*}
\Psi_n(y,x)=~&G(y,x)-G(0,x)-\left\langle\partial_yG(y,x)\mid_{y=0},y\right\rangle1_{\partial B(0,r)}(y)\\
&-\frac12\left\langle\partial_y^2G(y,x)\mid_{y=0}y,y\right\rangle1_{B(0,r)}(y).
\end{align*}
At this point, let us fix $\overline\theta\in(0,\frac{r}{2})$. By Lemma \ref{le3.1} and Lemma \ref{le4.2}, we find that,
\begin{equation}
\label{5.9}
\begin{aligned}
f_n^*(y)&=\rho_n\overline h_1|y|^{2\alpha}e^{\tilde u_n^{(1)}}(\xi_n(y)+O(\|\tilde u_n^{(1)}-\tilde u_n^{(2)}\|_{L^\infty(\Omega)})\\
&=\rho_n\overline h_1|y|^{2\alpha}e^{\tilde u_n^{(1)}}(-b_0+o(1)),
\end{aligned}
\end{equation}
for any $y\in\partial\Omega\setminus B(0,\overline\theta)$. By \eqref{2.3}, \eqref{2.11}, \eqref{2.12} and \eqref{5.9}, we conclude that
\begin{equation}
\label{5.10}
\begin{aligned}
f_n^*(y)=~&\rho_n\overline h_1|y|^{2\alpha}e^{\rho_{n,1}^{(1)}G(y,0)-\lambda_n-2\log(\gamma_n)-8\pi(1+\alpha)R(0,0)}(-b_0+o(1))\\
=~&(8(1+\alpha)^2)^2\frac{e^{-\lambda_n}}{\rho_n\overline h_1(0)}|y|^{2\alpha}e^{\Phi(y,0)}(-b_0+o(1))\quad \mathrm{for}\quad y\in\Omega\setminus B(0,\overline\theta),
\end{aligned}
\end{equation}
where
\[\Phi(y,0)=-4(1+\alpha)\log|y|+8\pi(1+\alpha)(R(y,0)-R(0,0))+\log(\overline h_1(0))-\log(\overline h_1(0)).\]
On the other hand, by \eqref{2.5}, we have for $y\in B(0,\overline\theta)$,
\begin{equation}
\label{5.11}
f_n^*(y)=\rho_nhe^{\tilde u_n^{(1)}}(\xi_n+O(\|\tilde u_n^{(1)}-\tilde u_n^{(2)}\|_{L^\infty(\Omega)}))=O\left(\frac{|y|^{2\alpha}e^{\lambda_n}}{(1+e^{\lambda_n}|y|^{2+2\alpha})^2}\right).
\end{equation}
Next, by \eqref{5.9}, for $y\in B(0,\overline\theta)$ and $x\in\partial B(0,r)$, we get
\begin{equation}
\label{5.12}
\Psi_n(y,x)=O\left(\frac{|y|^3}{|x|^3}\right),\quad \mathrm{and}\quad \nabla_x\Psi_n(y,x)=O\left(\frac{|y|^3}{|x|^4}\right).
\end{equation}
Let us define
\begin{equation}
\label{5.13}
\begin{aligned}
\overline G_n(x)=A_nG(0,x)+\sum_{h=1}^2B_{n,h}\partial_{y_h}G(y,x)\mid_{y=0}
+\frac12\sum_{h,k=1}^2C_{n,h,k}\partial_{y_hy_k}^2G(y,x)\mid_{y=0},
\end{aligned}
\end{equation}
so that, by \eqref{5.10}-\eqref{5.12}, we conclude that for $x\in\partial B(0,r)$, it holds
\begin{equation}
\label{5.14}
\begin{aligned}
\xi_n(x)+d_n-\overline G_n(x)&=\int_{\Omega\setminus B(0,\overline\theta)}\Psi_n(y,x)f_n^*(y)dy+\int_{B(0,\overline\theta)}\Psi_n(y,x)f_n^*(y)dy\\
&=-b_0\int_{\Omega\setminus B(0,\overline\theta)}\frac{(8(1+\alpha)^2)^2e^{-\lambda_n}}{\rho_n\overline h_1(0)}\Psi_n(y,x)|y|^{2\alpha}e^{\Phi(y,0)}dy\\
&~\quad+O(\int_{B(0,\overline\theta)}\frac{|y|^3}{|x|^3}\frac{|y|^{2\alpha}e^{\lambda_n}}
{(1+e^{\lambda_n|y|^{2+2\alpha}})^2}dy)+o(e^{-\lambda_n})\\
&=-b_0\int_{\Omega\setminus B(0,\overline\theta)}\frac{(8(1+\alpha)^2)^2e^{-\lambda_n}}{\rho_n\overline h_1(0)}\Psi_n(y,x)|y|^{2\alpha}e^{\Phi(y,0)}dy\\
&~\quad+O\left(\frac{m_{n,\alpha}}{|x|^3}\right)+o(e^{-\lambda_n})~\mathrm{in}~C^1(\partial B(0,r)),
\end{aligned}
\end{equation}
where
\begin{align*}
m_{n,\alpha}=
\begin{cases}
\sigma_n^3,\quad \mathrm{if}~2\alpha>1, \vspace{0.2cm}\\
\sigma_n^3\log(\sigma_n^{-1}),\quad \mathrm{if}~2\alpha=1, \vspace{0.2cm}\\
\sigma_n^{2+2\alpha}\overline\theta^{1-2\alpha},\quad \mathrm{if}~2\alpha<1.
\end{cases}
\end{align*}
Let us set
\begin{equation}
\label{5.15}
\zeta_n^*(x)=-b_0\int_{\Omega\setminus B(0,\overline\theta)}\frac{(8(1+\alpha)^2)^2e^{-\lambda_n}}{\rho_n\overline h_1(0)}\Psi_n(y,x)|y|^{2\alpha}e^{\Phi(y,0)}dy
\end{equation}
and then subsititute \eqref{5.14} into \eqref{5.7}, to derive that
\begin{equation}
\label{5.16}
\mathrm{L.H.S.~of}~\eqref{5.3}=\int_{\partial B(0,r)}4(1+\alpha)\left\langle\nu,D(\overline G_n+\zeta_n^*)(x)\right\rangle d\sigma+O(\sigma|A_n|)+O(\frac{m_{n,\alpha}}{r^3})+o(\sigma_n^2).
\end{equation}
To estimate the right hand side of \eqref{5.16}, we notice that for any pair of (smooth enough) functions $u$ and $v$, it holds
\begin{equation}
\label{5.17}
\begin{aligned}
&\Delta u(\nabla v\cdot x)+\Delta v(\nabla u\cdot x)\\
&=\mathrm{div}\left(\nabla u(\nabla v\cdot x)+\nabla v(\nabla u\cdot x)-\nabla u\cdot \nabla v (x)\right).
\end{aligned}
\end{equation}
In view of \eqref{5.13}, we also see that, for any $\underline{\theta}\in(0,r)$,
\begin{equation}
\label{5.18}
\Delta \overline G_n(x)=A_n=\int_{\Omega}f_n^*dy=\int_{\Omega}\frac{\rho_nh(e^{\tilde u_n^{(1)}}-e^{\tilde u_n^{(2)}})}{\|\tilde u_n^{(1)}-\tilde u_n^{(2)}\|_{L^\infty(\Omega)}}=0~\mathrm{for}~x\in B(0,r)\setminus B(0,\underline{\theta}),
\end{equation}
and moreover, by using \eqref{5.4} and \eqref{5.1}, we have
\begin{equation}
\label{5.19}
\Delta(G_n-\varphi_n)(x)=0\quad\mathrm{for}\quad x\in B(0,r)\setminus B(0,\underline{\theta}).
\end{equation}
By using \eqref{5.17}-\eqref{5.19} and \eqref{5.5}, we conclude that
\begin{align*}
0=&\int_{B(0,r)\setminus B(0,\underline{\theta})}\left[\Delta\overline G_n(\nabla(G_n-\varphi_n)\cdot x)+\Delta(G_n-\varphi_n)(\nabla\overline G_n\cdot x)\right]dx\\
=&\int_{\partial(B(0,r)\setminus B(0,\underline{\theta}))}\left(\frac{\partial\overline G_n}{\partial\nu}(\nabla(G_n-\varphi_n)\cdot x)+\frac{\partial(G_n-\varphi_n)}{\partial\nu}(\nabla\overline G_n\cdot x)-\nabla\overline G_n\cdot\nabla (G_n-\varphi_n)\left\langle x,\nu\right\rangle\right)d\sigma\\
=&-\frac{\rho_n}{2\pi}\int_{\partial(B(0,r)\setminus B(0,\underline{\theta}))}\frac{\partial\overline G_n}{\partial\nu}d\sigma,
\end{align*}
and thus,
\begin{equation}
\label{5.20}
\int_{\partial B(0,r)}\frac{\partial\overline G_n}{\partial\nu}(x)d\sigma
=\int_{\partial B(0,\underline{\theta})}\frac{\partial\overline G_n}{\partial\nu}(x)d\sigma.
\end{equation}
At this point, let us denote by $o_{\underline{\theta}}(1)$ any quantity which converges to $0$ as $\underline{\theta}\to 0^+$, and then observe that,
\begin{equation}
\label{5.21}
4(1+\alpha)\int_{\partial B(0,\underline{\theta})}\left\langle\nu, A_nD_xG(0,x)\right\rangle d\sigma=-4(1+\alpha)A_n+o_{\underline{\theta}}(1).
\end{equation}
Since, $D_iD_h\log|x|=\frac{\delta_{ih}|x|^2-2x_ix_h}{|x|^4}$, then we find that,
\begin{equation}
\label{5.22}
\int_{\partial B(0,\underline{\theta})}\left\langle\nu,D_x\partial_{y_h}(\log|y-x|)\mid_{y=0}\right\rangle d\sigma=-\int_{\partial B(0,\underline{\theta})}\sum_{i=1}^2\frac{x_i}{|x|}\left(\frac{\delta_{ih}|x|^2-2x_ix_h}{|x|^4}\right)d\sigma=0.
\end{equation}
We observe that, if $h=k$ then $D_i\log|x|=\frac{x_i}{|x|^2}$,
$$D_iD_{hh}^2\log|x|=-\frac{2x_i}{|x|^4}-\frac{4x_h\delta_{ih}}{|x|^4}+\frac{8x_h^2x_i}{|x|^6},$$
and thus,
\begin{equation}
\label{5.23}
\int_{\partial B(0,\underline{\theta})}\left\langle\nu, D_x\frac{\partial^2}{\partial y_h^2}\log\frac{1}{|y-x|}\mid_{y=0}\right\rangle d\sigma
=\int_{\partial B(0,\underline{\theta})}\left(\frac{2}{|x|^3}
-\frac{4x_h^2}{|x|^5}\right)d\sigma=0.
\end{equation}
If $h\neq k$, then
$$D_iD_{hk}^2\log|x|=-\frac{2(x_h\delta_{ki}+x_k\delta_{hi})}{|x|^4}+\frac{8x_kx_ix_h}{|x|^6},$$
which implies that
\begin{equation}
\label{5.24}
\int_{\partial B(0,\underline{\theta})}\left\langle\nu, D_x\frac{\partial^2}{\partial y_h\partial y_k}\log\frac{1}{|y-x|}\mid_{y=0}\right\rangle d\sigma=\int_{\partial B(0,\underline{\theta})}(\frac{4x_hx_k}{|x|^5}-\frac{8x_hx_k}{|x|^5})d\sigma=0.
\end{equation}
By \eqref{5.20}-\eqref{5.24}, we conclude that
\begin{equation}
\label{5.25}
4(1+\alpha)\int_{\partial B(0,r)}\left\langle\nu,D_x\overline G_n(x)\right\rangle d\sigma=-4(1+\alpha)A_n+o_{\underline{\theta}}(1).
\end{equation}
Next we estimate the other terms in \eqref{5.16}, that is $4(1+\alpha)\int_{\partial B(0,r)}\left\langle\nu,D_x\zeta_n^*(x)\right\rangle d\sigma$, where $\zeta_n^*$ is defined in \eqref{5.15}. Clearly we have
\begin{align*}
D_x\Psi_n(y,x)=~&D_x\left(G(y,x)-G(0,x)-\left\langle\partial_yG(y,x)\mid_{y=0},y\right\rangle1_{B(0,r)}(y)\right)\\
&-\frac12D_x\left(\frac12\left\langle\partial_y^2G(y,x)\mid_{y=0}(y),y\right\rangle1_{B(0,r)}(y)\right).
\end{align*}
If $y\in\Omega\setminus B(0,\overline\theta)$ and $x\in\partial B(0,\theta)$ with $\theta\ll(\overline\theta)^2$, then we find that
\begin{equation}
\label{5.26}
|D_xG(x,y)|\leq\frac{C}{\sqrt{\theta}} \mbox{ for some constant } C>0,
\end{equation}
which implies
$$\int_{\partial B(0,\theta)}\left\langle\nu,D_xG(y,x)\right\rangle dx=o_{\theta}(1).$$
Thus \eqref{5.22}-\eqref{5.24} and \eqref{5.26} imply that
\begin{equation}
\label{5.27}
\begin{aligned}
&4(1+\alpha)\int_{\partial B(0,\theta)}\left\langle\nu,D_x\Psi_n(y,x)\right\rangle dx\\
&=-4(1+\alpha)\int_{\partial B(0,\theta)}\left\langle\nu,D_xG(0,x)\right\rangle dx\\
&\quad-4(1+\alpha)\int_{\partial B(0,\theta)}\left\langle\nu,D_x\left\langle\partial_yG(y,x)\mid_{y=0},y\right\rangle1_{B(0,r_0)}(y)\right\rangle dx\\
&\quad-2(1+\alpha)\int_{\partial B(0,\theta)}\left\langle\nu,D_x\left\langle\partial_y^2G(y,x)\mid_{y=0}y,y\right\rangle1_{B(0,r_0)}(y)\right\rangle dx+o_{\theta}(1)\\
&=4(1+\alpha)+o_{\theta}(1)\quad\mathrm{for}\quad y\in\Omega\setminus B(0,\overline\theta),\quad \mathrm{and}\quad x\in\partial B(0,\theta).
\end{aligned}
\end{equation}
We observe that
$$-\Delta_x\Psi_n(y,x)=\delta_y\quad\mathrm{for}\quad x\in B(0,r)\setminus B(0,\theta)$$
and let us choose $u(x)=\Psi_n(y,x)$ and $v(x)=G_n(x)-\varphi_n(x)$ in \eqref{5.17}. Then we consider the following two cases:
\begin{itemize}
\item [(i)] If $y\in\partial B(0,r)\setminus B(0,\overline\theta)$, then from \eqref{5.17} and \eqref{5.27}, we obtain that
\begin{equation}
\label{5.28}
\begin{aligned}
&4(1+\alpha)\int_{\partial B(0,r)}\left\langle\nu,D_x\Psi_n(y,x)\right\rangle dx\\
&=4(1+\alpha)\int_{\partial B(0,\theta)}\left\langle\nu,D_x\Psi_n(y,x)\right\rangle dx-4(1+\alpha)=o_\theta(1).
\end{aligned}
\end{equation}
\item [(ii)] If $y\in\Omega\setminus B(0,r)$, then we see from \eqref{5.17} and \eqref{5.27} that
\begin{equation}
\label{5.29}
\begin{aligned}
4(1+\alpha)\int_{\partial B(0,r)}\left\langle\nu,D_x\Psi_n(y,x)\right\rangle dx
=~&4(1+\alpha)\int_{\partial B(0,\theta)}\left\langle\nu,D_x\Psi_n(y,x)\right\rangle dx\\
=~&4(1+\alpha)+o_\theta(1).
\end{aligned}
\end{equation}
\end{itemize}
and by \eqref{5.15}, and \eqref{5.28}-\eqref{5.29}, we finally conclude that
\begin{equation}
\label{5.30}
\begin{aligned}
&4(1+\alpha)\int_{\partial B(0,r)}\left\langle\nu,D_x\zeta_n^*(x)\right\rangle dx\\
&=-\frac{(8(1+\alpha)^2)^3b_0e^{-\lambda_n}}{2\rho_n\overline h_1(0)}\int_{\Omega\setminus B(0,\overline\theta)}
\left(\int_{\partial B(0,r)}\left\langle\nu,D_x\Psi_n\right\rangle\right)|y|^{2\alpha}e^{\Phi(y,0)}dxdy\\
&=-\frac{(8(1+\alpha)^2)^3b_0e^{-\lambda_n}}{2\rho_n\overline h_1(0)}\int_{\Omega\setminus B(0,r)}|y|^{2\alpha}e^{\Phi(y,0)}dy+o(e^{-\lambda_n}).
\end{aligned}
\end{equation}
Obviously from \eqref{5.16}, \eqref{5.25} and \eqref{5.30} we get the conclusion of Lemma \ref{le5.2}
\end{proof}

To estimate the right hand side of \eqref{5.3} of Lemma \ref{le5.1}, we recall, see for example \eqref{5.9}, that
\[f_n^*(x)=\rho_nh(x)e^{\tilde u_n^{(1)}}(\xi_n+o(1)).\]
Recall also the definitions of $\Phi(x,0)$ and $\mathcal{H}_p=\mathcal{H}_0$ in \eqref{P} and \eqref{H}, respectively and the definition of $\ell(p)$ after \eqref{2.3}. A crucial point in our proof is the following estimate.

\begin{lemma}
\label{le5.3}
\begin{itemize}
\item [(i)]
\begin{align*}
\int_{\partial B(0,r)}rf_n^*d\sigma=~&-\frac{128(1+\alpha)^4b_0\pi e^{-\lambda_n}}{\rho_n\overline h_1(0)r^{2+2\alpha}}-\frac{32(1+\alpha)^4b_0\pi e^{-\lambda_n}}{\rho_n\overline h_1(0)r^{2\alpha}}\Delta\log h_*(0)\\
&+O(r^{1-2\alpha}e^{-\lambda_n})+\frac{o(e^{-\lambda_n})}{r^{2+2\alpha}},
\end{align*}

\item [(ii)]
\begin{align*}
\int_{B(0,r)}f_n^*(x)dx=\frac{64(1+\alpha)^4b_0e^{-\lambda_n}}{\rho_n\overline h_1(0)}\int_{\Omega\setminus B(0,r)}|x|^{2\alpha}e^{\Phi(x,0)}dx+\frac{o(e^{-\lambda_n})}{r^{2+2\alpha}},
\end{align*}

\item [(iii)]
\begin{align*}
&\int_{B(0,r)}f_n^*\left\langle D(\log\overline h_1+\varphi_n),x\right\rangle dx\\
&=-2b_0\ell(p)\sigma_n^2+o(\sigma_n)|\nabla\mathcal{H}_0(0)|+O(m_{n,1}(\alpha))\\
&\quad+O(\sigma_n^{2\epsilon_0}+\lambda_n\sigma_n^2)
\left(|\nabla\mathcal{H}_0(0)|\sg_n+\sg_n^2\right)+O(\sigma_n^{2+\epsilon_0})\\
&\quad+\left(O(R^{-2\alpha})+O(\lambda_n)|A_n|+O(\frac{1}{R})\right)\left(|\nabla\mathcal{H}_0(0)|\sigma_n+\sigma_n^2\right).
\end{align*}
\end{itemize}
\noindent where $O(m_{n,1}(\alpha))$ is defined after \eqref{5.35} and $O(1)$ is used to denote any quantity uniformly bounded with respect to $r$, $R$ and $n.$
\end{lemma}

\begin{proof}
(i) We first observe that \eqref{5.10} implies that
\begin{equation}
\label{5.31}
\int_{\partial B(0,r)}rf_n^*(x)d\sigma=\int_{\partial B(0,r)}\frac{(8(1+\alpha)^2)^2e^{-\lambda_n}(-b_0+o(1))|x|^{2\alpha}e^{\mathcal{H}_0(x)}}{\rho_n\overline h_1(0)|x|^{3+4\alpha}}d\sigma.
\end{equation}	
Clearly we have
\begin{equation}
\label{5.32}
\mathcal{H}_0(x)=\left\langle D\mathcal{H}_0(0),x\right\rangle+\frac12\left\langle D_x^2\mathcal{H}_0\mid_{x=0}x,x\right\rangle+O(|x|^3).
\end{equation}
By \eqref{5.31} and \eqref{5.32}, we obtain,
\begin{align*}
&\int_{\partial B(0,r)}rf_n^*(x)d\sigma\\
&=\int_{\partial B(0,r)}\frac{(8(1+\alpha)^2)^2e^{-\lambda_n}}{\rho_n\overline h_1(0)|x|^{3+3\alpha}}
\left(b_0(1+\left\langle D\mathcal{H}_0,x\right\rangle+\frac12\left\langle D_x^2\mathcal{H}_0\mid_{x=0}x,x\right\rangle)+O(|x|^3)+o(1)\right)d\sigma\\
&=-\int_{\partial B(0,r)}\frac{(8(1+\alpha)^2)^2e^{-\lambda_n}b_0(1+\frac{\Delta \mathcal{H}_0}{4}|x|^2)}{\rho_n\overline h_1(0)|x|^{3+2\alpha}}d\sigma+O(r^{1-2\alpha}e^{-\lambda_n})+\frac{o(e^{-\lambda_n})}{r^{2+2\alpha}}\\
&=-\frac{128(1+\alpha)^4b_0\pi e^{-\lambda_n}}{\rho_n\overline h_1(0)r^{2+2\alpha}}-\frac{32(1+\alpha)^4b_0\pi e^{-\lambda_n}}{\rho_n\overline h_1(0)r^{2\alpha}}\Delta\log(h_*(0))+O(r^{1-2\alpha}e^{-\lambda_n})+\frac{o(e^{-\lambda_n})}{r^{2+2\alpha}},
\end{align*}
which proves (i).

(ii) We notice that $A_n=\int_{\Omega}f_n^*=0$, and thus
\begin{equation}
\label{5.33}
\int_{B(0,r)}f_n^*(x)dx=-\int_{\Omega\setminus B(0,r)}f_n^*(x)dx.
\end{equation}
By \eqref{5.10} we see that
\begin{equation}
\label{5.34}
\begin{aligned}
-\int_{\Omega\setminus B(0,r)}f_n^*dx&=\int_{\Omega\setminus B(0,r)}\frac{64(1+\alpha)^4b_0e^{-\lambda_n}}{\rho_n\overline h_1(0)}|x|^{2\alpha}e^{\Phi(x,0)}dx+\frac{o(e^{-\lambda_n})}{r^{2+2\alpha}}\\
&=\frac{64(1+\alpha)^4b_0e^{-\lambda_n}}{\rho_n\overline h_1(0)}\int_{\Omega\setminus B(0,r)}|x|^{2\alpha}e^{\Phi(x,0)}dx+\frac{o(e^{-\lambda_n})}{r^{2+2\alpha}},
\end{aligned}
\end{equation}
which proves (ii).

(iii) By \eqref{2.2} and \eqref{2.5}, we see that
\begin{equation*}
\tilde u_n(x)=U_n(x)+8\pi(1+\alpha)(R(x,0)-R(0,0))+\eta_n(x),\quad x\in B(0,r),
\end{equation*}
where
$$\eta_n(x)=\sigma_n\psi_{n,1}(\sigma_n^{-1}x)+\sigma_n^2\psi_{n,2}(\sigma_n^{-1}x)+O(\sigma_n^2),$$
see \eqref{2.6}, \eqref{2.7} and \eqref{2.9}. Thus, we set
$$\omega_n(r)=\|\tilde u_n^{(1)}-\tilde u_n^{(2)}\|_{L^\infty(B(0,r))},$$
and use Lemma \ref{le3.1} and \eqref{2.3}, we deduce that
\begin{equation}
\label{5.35}
\begin{aligned}
&\int_{B(0,r)}f_n^*\left\langle D(\log\overline h_1(x)+\varphi_n),x\right\rangle dx\\
&=\int_{B(0,r)}\frac{\rho_n\overline h_1(0)|x|^{2\alpha}e^{\lambda_n+\mathcal{H}_0(x)+\eta_n(x)}}{(1+\gamma_ne^{\lambda_n}|x|^{2+2\alpha})^2}
(\xi_n-\frac{\omega_n(r)}{2}\xi_n^2
+O(\omega_n^2(r)))\left\langle D\mathcal{H}_0(x),x\right\rangle dx\\
&=\int_{B(0,r)}\frac{\rho_n\overline h_1(0)|x|^{2\alpha}e^{\lambda_n+\mathcal{H}_0(x)+\eta_n(x)}}{(1+\gamma_ne^{\lambda_n}|x|^{2+2\alpha})^2}
(\xi_n-\frac{\omega_n(r)}{2}\xi_n^2
+O(\omega_n^2(r)))\\
&~\quad\quad~\times\left\langle D\mathcal{H}_0(0)+D^2\mathcal{H}_0(0)x+O(|x|^2),x\right\rangle dx\\
&=\int_{B(0,\sigma_n^{-1}r)}\frac{\rho_n\overline h_1(0)|z|^{2\alpha}}{(1+\gamma_n|z|^{2+2\alpha})^2}
\left(\hat\xi_n-\frac{\omega_n}{2}\hat\xi_n^2+\left\langle D\mathcal{H}_0(0),\sigma_nz\right\rangle+\eta_n+O(\sigma_n^2|z|^2)+O(\omega_n^2)\right)\\
&~\quad\quad~\times\left\langle D\mathcal{H}_0(0)+D^2\mathcal{H}_0(0)\cdot\sigma_nz+O(\sigma_n^2|z|^2),\sigma_nz\right\rangle dz=:K_{n,r}.
\end{aligned}
\end{equation}
Set
$$
m_{n,1}(\al)=\graf{\sg_n^3\,\mbox{ if } \al> \frac12\\ \log(r\sg_n^{-1})\sg_n^3\, \mbox{ if } \al=\frac12\\r^{1-2\al}\sg_n^{2(1+\al)}
\,\mbox{ if } \al\in(0,\frac12)},\quad
m_{n,2}(\al)=\graf{\sg_n^4\,\mbox{ if } \al> 1\\ \log(r\sg_n^{-1})\sg_n^4\, \mbox{ if } \al=1\\r^{2-2\al}\sg_n^{2(1+\al)}
\,\mbox{ if } \al\in(0,1)},
$$
$$
m_{n,3}(\al)=\graf{\sg_n^5\,\mbox{ if } \al> \frac32\\ \log(r\sg_n^{-1})\sg_n^5\, \mbox{ if } \al=\frac32\\
r^{3-2\al}\sg_n^{2(1+\al)} \,\mbox{ if } \al\in(0,\frac32)}.
$$
Using \eqref{5.35} together with \eqref{2.4}, \eqref{2.6} and Lemma \ref{le3.1}, we conclude that
{\allowdisplaybreaks
\begin{align*}
K_{n,r}&=\int_{B(0,r\sg_n^{-1})}\frac{\rho_n \ov{h}_1(0)|z|^{2\al}  }{(1+\gm_n|z|^{2(1+\al)})^2}\\
& \times\Big(\hat{\xi}_n
{{-\frac{\omega_n(r)}{2}(\hat{\xi}_n)^2}}
+\sg_n(\psi_{n,1}(z)+<D \mathcal{H}_{{0}}(0),z>)+O(\sg_n^2|z|^2)+
O((\sg_n^{2\epsilon_0}+\lm_n\sg_n^{2})^2)+\sg_n^2\tilde{\psi}_{n,2}(z)\Big)\\
&  \times \Big<D\mathcal{H}_{{0}}(0)+   D^2\mathcal{H}_{{0}}(0) \cdot \sg_n z
+O(\sg_n^2|z|^2),  \sg_n z\Big> {d} z
\\&=\int_{B(0,r\sg_n^{-1})}\frac{\rho_n \ov{h}_1(0)|z|^{2\al} \Big(\hat{\xi}_n
{{-\frac{\omega_n(r)}{2}(\hat{\xi}_n)^2}}+\sg_n(\psi_{n,1}(z)+
<D \mathcal{H}_{{0}}(0),z>)\Big) }{(1+\gm_n|z|^{2(1+\al)})^2}\\& \times
   < D\mathcal{H}_{{0}}(0)+\sg_n D^2\mathcal{H}_{{0}}(0)\cdot  z,z>\sg_n{d} z
 {\,+\,O(m_{n,1}(\al)+m_{n,2}(\al)+m_{n,3}(\al))}\\&
 +\left(|\nabla\mathcal{H}_0(0)|\sg_n+\sg_n^2\right)O((\sg_n^{2\epsilon_0}+\lm_n\sg_n^{2})^2)\\
&=I_{n,1}+I_{n,2}+{O(m_{n,1}(\al))}+\left(|\nabla\mathcal{H}_0(0)|\sg_n+\sg_n^2\right)O(\sg_n^{4\epsilon_0}),
\end{align*}}
where
\begin{align*}
I_{n,1}=\int_{B(0,r\sigma_n^{-1})}\frac{\rho_n\overline h_1(0)|z|^{2\alpha}(\hat\xi_n-\frac{\omega_n}{2}\hat\xi_n^2)}{(1+\gamma_n|z|^{2+2\alpha})^2}\left\langle D\mathcal{H}_0(0)+\sigma_nD^2\mathcal{H}_0(0)\cdot z,z\right\rangle\sigma_ndz,
\end{align*}
and
\begin{equation*}
I_{n,2}=\int_{B(0,r\sigma_n^{-1})}\frac{\rho_n\overline h_1(0)|z|^{2\alpha}(\psi_{n,1}(z)+\left\langle D\mathcal{H}_0(0),z\right\rangle)}
{(1+\gamma_n|z|^{2+2\alpha})^2}\left\langle D\mathcal{H}_0(0),z\right\rangle\sigma_n^2dz.
\end{equation*}
In view of \eqref{2.8}, \eqref{2.9}, \eqref{2.10} and \eqref{2.3}, we have
\begin{equation*}
\left\langle D\mathcal{H}_0(0),z\right\rangle =\partial_{x_1}\mathcal{H}_0(0)z_1+{O(\sigma_n^2)(z_1+z_2)}=a_{n,1}z_1+O(\sigma_n^2)(z_1+z_2),
\end{equation*}
and then, putting $a_1=\partial_{x_1}\mathcal{H}_0$ and $\Lambda(z)=\rho_n\overline h_1(0)|z|^{2\alpha}$, we conclude that
\begin{align*}
\sigma_n^{-2}I_{n,2}=&\int_{B(0,r\sigma_n^{-1})}\frac{\Lambda(z)(\psi_{n,1}(z)+a_1z_1+O(\sigma_n^2)(z_1+z_2))}
{(1+\gamma_n|z|^{2+2\alpha})^2}(a_1z_1+O(\sigma_n^2)(z_1+z_2))dz\\
=&\int_{B(0,r\sigma_n^{-1})}\frac{\Lambda(z)}{(1+\gamma_n|z|^{2+2\alpha})^2}
\left(-\frac{2(1+\alpha)a_1z_1}{\alpha(1+\gamma_n|z|^{2+2\alpha})}+a_1z_1+O(\sigma_n^2)(z_1+z_2)\right)\\
&\times (a_1z_1+O(\sigma_n^2)(z_1+z_2))dz\\
=&-\int_{B(0,r\sigma_n^{-1})}\frac{2(1+\alpha)\Lambda(z)a_1^2z_1^2}{\alpha(1+\gamma_n|z|^{2+2\alpha})^3}dz
+\int_{B(0,r\sigma_n^{-1})}\frac{\Lambda(z)a_1^2z_1^2}{(1+\gamma_n|z|^{2+2\alpha})^2}dz+O(\sigma_n^2)\\
=&-\rho_n\overline h_1(0)\frac{a_1^2\pi^2}{2(1+\alpha)^2\gamma_n^{\frac{2+\alpha}{1+\alpha}}}
\frac{1}{\sin\frac{\pi}{1+\alpha}}+\rho_n\overline h_1(0)\frac{a_1^2\pi}
{2\alpha\gamma_n^{\frac{2+\alpha}{1+\alpha}}}\Gamma(\frac{2+\alpha}{1+\alpha})\Gamma(\frac{1+2\alpha}{1+\alpha})\\
&+O(\sigma_n^{2\alpha})+O(\sigma_n^2)\\
=&~O(\sigma_n^{2\alpha})+O(\sigma_n^2),
\end{align*}
where we used the properties of $\Gamma(x)$, and thus
\begin{equation}
\label{5.36}
I_{n,2}=O(\sigma_n^{2+\epsilon_0}).
\end{equation}
On the other hand, in view of Lemma \ref{le4.1}, for any fixed $R\geq 1$ large, we have
\begin{align*}
&\int_{B(0,R)}\frac{\Lambda(z)(\hat\xi_n-\frac{\omega_n}{2}\hat\xi_n^2)}{(1+\gamma_n|z|^{2+2\alpha})^2}\left\langle D\mathcal{H}_0(0)+\sigma_nD^2\mathcal{H}_0(0)\cdot z,z\right\rangle \sg_n dz\\
&=\int_{B(0,R)}\frac{\Lambda(z)(b_0\hat\xi_0(z)+o(1)+O(\sigma_n^{2\epsilon_0}+\lambda_n\sigma_n^2))}
{(1+\gamma_n|z|^{2+2\alpha})^2}\left\langle D\mathcal{H}_0(0)
+\sigma_nD^2\mathcal{H}_0(0)\cdot z,z\right\rangle \sg_n dz\\
&=\sigma_n^2\int_{B(0,R)}\frac{\Lambda(z)(b_0\hat\xi_0(z)+O(\sigma_n^{2\epsilon_0}+\lambda_n\sigma_n^2))}
{(1+\gamma_n|z|^{2+2\alpha})^2}\left\langle D^2\mathcal{H}_0(0)\cdot z,z\right\rangle dz\\
&\quad+o(\sg_n)|\nabla\mathcal{H}_0(0)|+O(\sigma_n^{2\epsilon_0}+\lambda_n\sigma_n^2)
\left(|\nabla\mathcal{H}_0(0)|\sg_n+\sg_n^2\right)\\
&=\sigma_n^2\int_{B(0,R)}\frac{\Lambda(z)(b_0\hat\xi_0(z)}{(1+\gamma_n|z|^{2+2\alpha})^2}\left\langle D^2\mathcal{H}_0(0)\cdot z,z\right\rangle dz+o(\sg_n)|\nabla\mathcal{H}_0(0)| \\
&\quad+O(\sigma_n^{2\epsilon_0}+\lambda_n\sigma_n^2)
\left(|\nabla\mathcal{H}_0(0)|\sg_n+\sg_n^2\right)\\
&=8\pi(1+\alpha)\overline h_1(0)b_0\sigma_n^2
\int_{B(0,R)}\frac{|z|^{2\alpha}\hat\xi_0(z)}{(1+\gamma|z|^{2+2\alpha})^2}\left\langle D^2\mathcal{H}_0(0)\cdot z,z\right\rangle dz\\
&\quad+o(\sg_n)|\nabla\mathcal{H}_0(0)|+O(\sigma_n^{2\epsilon_0}+\lambda_n\sigma_n^2)
\left(|\nabla\mathcal{H}_0(0)|\sg_n+\sg_n^2\right).
\end{align*}
Finally we have
\begin{align*}
&\int_{B(0,R)}\frac{|z|^{2\alpha}\hat\xi_0(z)}{(1+\gamma|z|^{2+2\alpha})^2}\left\langle D^2\mathcal{H}_0(0)\cdot z,z\right\rangle dz=\frac{\Delta\mathcal{H}_0(0)}{2}\int_{B(0,R)}\frac{1-\gamma|z|^{2+2\alpha}}
{(1+\gamma|z|^{2+2\alpha})^3}|z|^{2\alpha+2}dz\\
&=-\frac{\Delta\mathcal{H}_0(0)}{2}\frac{\pi^2}{(1+\alpha)^3\gamma^{\frac{2+\alpha}{1+\alpha}}
\sin\frac{\pi}{1+\alpha}}+O(R^{-2\alpha})\\
&=-\frac{\pi}{2(1+\alpha)^2\overline h_1(0)\gamma^{\frac{1}{1+\alpha}}\sin\frac{\pi}{1+\alpha}}\Delta\log(h_*(0))+O(R^{-2\alpha}).
\end{align*}
On the other side, in view of \eqref{4.8}, we also see that if $R\leq |z|\leq r/\sigma_n$, then it holds
\begin{equation}
\label{5.37}
\hat\xi_n(z)=-d_n+O(\lambda_n)|A_n|+O(\frac{1}{|z|}),
\end{equation}
and thus
$$\hat\xi_n(z)^2=d_n^2+O(\lambda_n^2+\frac{\lambda_n}{|z|})|A_n|+O(\frac{1}{|z|^2}).$$
As a consequence, by Lemma \ref{le3.1}, we find that
\begin{align*}
&\int_{B(0,r/\sigma_n)\setminus B(0,R)}\frac{\rho_n\overline h_1(0)|z|^{2\alpha}(\hat\xi_n-\frac{\omega_n}{2}\hat\xi_n^2)}{(1+\gamma_n|z|^{2+2\alpha})^2}\left\langle D\mathcal{H}_0(0)+\sigma_nD^2\mathcal{H}_0(0)\cdot z,z\right\rangle \sg_n dz\\
&=\left(-d_n-\frac{\omega_n(r)}{2}d_n^2\right)\int_{B(0,r/\sigma_n)\setminus B(0,R)}
\frac{\rho_n\overline h_1(0)|z|^{2\alpha}}{(1+\gamma_n|z|^{2+2\alpha})^2}\left\langle D\mathcal{H}_0(0)+\sigma_nD^2\mathcal{H}_0(0)\cdot z,z\right\rangle \sg_n dz\\
&\quad+\int_{B(0,r/\sigma_n)\setminus B(0,R)}
\frac{\rho_n\overline h_1(0)|z|^{2\alpha}(O(\lambda_n)|A_n|+O(\frac{1}{|z|}))}
{(1+\gamma_n|z|^{2+2\alpha})^2}\left\langle D\mathcal{H}_0(0)+\sigma_nD^2\mathcal{H}_0(0)\cdot z,z\right\rangle \sg_n dz\\
&=-b_0\Delta\log h_*(0)\int_{B(0,r/\sigma_n)\setminus B(0,R)}\frac{\rho_n\overline h_1(0)|z|^{2\alpha+2}}{(1+\gamma_n|z|^{2+2\alpha})^2}\sigma_n^2 dz+O(\sigma_n^2(\sg_n^{2\epsilon_0}+\lambda_n\sigma_n^2))\\
&\quad+\left(O(|\lambda_n|)|A_n|+O\left(\frac{1}{R}\right)\right)\left(|\nabla\mathcal{H}_0(0)|\sg_n+\sigma_n^2\right)\\
&=O(R^{-2\alpha})\sigma_n^2+\left(O(R^{-2\alpha})+O(\lambda_n)|A_n|+O\left(\frac{1}{R}\right)\right)
\left(|\nabla\mathcal{H}_0(0)|\sg_n+\sigma_n^2\right)\\
&\quad+O(\sigma_n^{2+2\epsilon_0}+\lambda_n\sigma_n^4)\\
\end{align*}
Collecting the above estimates we conclude that
\begin{align*}
&\int_{B(0,r)}f_n^*\left\langle D(\log\overline h_1+\varphi_n),x\right\rangle dx\\
&=-2b_0\ell(p)\sigma_n^2+o(\sigma_n)|\nabla\mathcal{H}_0(0)|+O(m_{n,1}(\alpha))
+O(\sigma_n^{2\epsilon_0}+\lambda_n\sigma_n^2)
\left(|\nabla\mathcal{H}_0(0)|\sg_n+\sg_n^2\right)\\
&\quad+O(\sigma_n^{2+\epsilon_0})+\left(O(R^{-2\alpha})+O(\lambda_n)|A_n|+O(\frac{1}{R})\right)\left(|\nabla\mathcal{H}_0(0)|\sigma_n+\sigma_n^2\right).
\end{align*}
\end{proof}

\

Recall that $p=0$. Using the assumptions $\ell(p)\neq 0$ and $\nabla\mathcal{H}_0(0)=0$ we can now prove that $b_0=0$.

\begin{lemma}
\label{le5.4}
$b_0=0.$
\end{lemma}

\begin{proof}
By \eqref{5.3} and Lemmas \ref{le5.2}-\ref{le5.3}, we have for any $r\in(0,1)$ and $R>1$, 
\begin{align*}
&-4(1+\alpha)A_n-\frac{(8(1+\alpha)^2)^3b_0e^{-\lambda_n}}{2\rho_n\overline h_1(0)}
\int_{\Omega\setminus B(0,r)}|y|^{2\alpha}e^{\Phi(y,0)}dy+o(\sigma_n^2)
+O(\sigma_n|A_n|+\frac{\sigma_n^3}{r^3})\\
&=-\frac{128(1+\alpha)^4b_0\pi e^{-\lambda_n}}{\rho_n\overline h_1(0)r^{2+2\alpha}}
-\frac{32(1+\alpha)^4b_0\pi e^{-\lambda_n}}{\rho_n\overline h_1(0)r^{2\alpha}}\Delta\log h_*(0)+O(r^{1-2\alpha}e^{-\lambda_n})\\
&\quad-\frac{128(1+\alpha)^5b_0e^{-\lambda_n}}{\rho_n\overline h_1(0)}\int_{\Omega\setminus B(0,r)}|y|^{2\alpha}e^{\Phi(y,0)}dy+\frac{o(e^{-\lambda_n})}{r^{2+2\alpha}}
+2b_0\ell(p)\sigma_n^2\\
&\quad+o(\sigma_n)|\nabla\mathcal{H}_0(0)|+O(m_{n,1}(\alpha))
+O(\sigma_n^{2\epsilon_0}+\lambda_n\sigma_n^2)
\left(|\nabla\mathcal{H}_0(0)|\sg_n+\sg_n^2\right)\\
&\quad+O(\sigma_n^{2+\epsilon_0})+\left(O(R^{-2\alpha})+O(\lambda_n)|A_n|+O(\frac{1}{R})\right)\left(|\nabla\mathcal{H}_0(0)|\sigma_n+\sigma_n^2\right).
\end{align*}	
Recall $A_n=0$. Since $\nabla\mathcal{H}_0(0)=0$ by assumption, after some manipulations, for $r\in(0,r_0)$ and any $R>1$, we find that
\begin{align*}
b_0\ell(p)\sigma_n^2&=o(\sigma_n^{2})+O(m_{n,1}(\alpha))
+O(R^{-2\alpha}+R^{-1})\sigma_n^2+O(\sigma_n^{2+2\epsilon_0}+\lambda_n \sigma_n^4)\\
&\quad+O\left(\frac{\sigma_n^3}{r^3}\right)+\frac{o(e^{-\lambda_n})}{r^{2+2\alpha}},
\end{align*}
which implies
\[b_0=0.\]
provided $\ell(p)\neq0$. Hence we finish the proof.
\end{proof}

\

\begin{proof}[Proof of Theorem \ref{th.main}]
Let $x_n^*$ be a maximum point of $\xi_n$, then we have,
\begin{equation}
\label{5.38}
|\xi_n(x_n^*)|=1.
\end{equation}
By Lemma \ref{le4.2} and Lemma \ref{le5.4} we have that $x_n^*\to p$. By Lemma \ref{le5.4}, it holds that
\begin{equation}
\label{5.39}
\lim_{n\to+\infty}e^{\frac{\lambda_n^{(1)}}{2(1+\alpha)}}s_n=+\infty,\quad
\mbox{where}~s_n=|x_n^*-p|.
\end{equation}
Setting $\bar\xi_n(x)=\xi(s_nx+p)$, then we have $\bar\xi_n$ satisfies
\begin{equation*}
\begin{aligned}
0&=\Delta\bar\xi_n+\rho_ns_n^2h(s_nx+p)c_n(s_nx+p)
\bar\xi_n\\
&=\Delta\bar\xi_n+\frac{\rho_n\overline{h}_1(p)|x|^{2\alpha}s_n^{2+2\alpha}
e^{\lambda_n^{(1)}}(1+O(s_n|x|)+o(1))\bar\xi_n}
{(1+\frac{\rho_n\bar h_1(p)}{8(1+\alpha)^2}e^{\lambda_n^{(1)}}|s_nx|^{2+2\alpha})^2}.
\end{aligned}
\end{equation*}
On the other hand, by \eqref{5.38}, we also have
\begin{equation}
\label{5.40}
\left|\bar\xi_n\left(\frac{x_n^*-p}{s_n}\right)\right|=|\xi_n(x_n^*)|=1.
\end{equation}
In view of \eqref{5.39} and $|\bar\xi_n|\leq 1$ we see that $\bar\xi_n\to \bar\xi_0$ on any compact subset of $\mathbb{R}^2\setminus\{0\}$, where $\bar\xi_0$ satisfies $\Delta\bar\xi_0=0$ in $\mathbb{R}^2\setminus\{0\}$. Since $|\bar\xi_0|\leq 1$, we have $\Delta\bar\xi_0=0$ in $\mathbb{R}^2$, which implies $\bar\xi_0$ is a constant. At this point, since $\frac{|x_n^*-p|}{s_n}=1$ and in view of \eqref{5.40}, we find $\bar\xi_0=1$ or $\bar\xi_0=-1$. From which we have $|\bar\xi_n(x)|\geq\frac12$ when $s_n\leq |x-p|\leq\frac{1}{2}s_n$, which contradicts to \eqref{4.5}-\eqref{4.7} since $e^{-\frac{\lambda_n^{(1)}}{2(1+\alpha)}}\ll s_n$ and $\lim_{n\to+\infty}s_n=0$ and $b_0=0$. This fact concludes the proof of Theorem \ref{th.main}.
\end{proof}

\

\section{The proof of Theorem \ref{th.main2}} \label{sec:non-deg}
In this section we give the proof of the non-degeneracy result stated in Theorem~\ref{th.main2}. Since the argument is similar to the one yielding local uniqueness of bubbling solutions we will be sketchy to avoid repetitions, referring to \cite{bjly2} for full details.

Suppose by contradiction the linearized problem \eqref{non-deg} admits a non-trivial solution $\phi_n$, where  $u_n$ is a singular $1$-bubbling solution of $\prn$ blowing up at the point $p_i$ for some $i\in\{1,\cdots,N\}$. We suppose with no loss of generality that $p_i=0\in\Omega$, set $\alpha_i=\alpha$ and
\begin{align*} 
\tilde u_n=u_n-\log\left(\int_{\Omega}he^{u_n}dx\right),\quad
\lambda_n=\max_{\Omega}\tilde u_n,\quad \sigma_n^{2(1+\alpha)}=e^{-\lambda_n},
\end{align*} 
Define
$$
\Xi_n = \dfrac{\phi_n-\frac{\int_\Omega he^{u_n}\phi_n\,dx}{\int_\Omega he^{u_n}\,dx}}{\left\|\phi_n-\frac{\int_\Omega he^{u_n}\phi_n\,dx}{\int_\Omega he^{u_n}\,dx}\right\|_{L^{\infty}(\Omega)}}\,,
$$
which plays the role of the difference of two bubbling solutions, see \eqref{xi} in the proof of Theorem \ref{th.main}. Then, $\Xi_n$ satisfies 
\begin{align} \label{equa}
\begin{cases}
\Delta\,\Xi_n+\rho_nh(x)c_n(x)\,\Xi_n(x)=0\quad &\mathrm{in}~\Omega,\\
\\
\Xi_n=-d_n\quad &\mathrm{on}~\partial\Omega,
\end{cases}
\end{align}
for some constant $d_n$ satisfying $|d_n|\leq 1$ and $c_n(x)=e^{\tilde{u}_n(x)}$. 

\medskip

\noindent \textbf{Step 1.} We start by considering the asymptotic behavior of $\Xi_n$ near the blow up point $p_i$. After a suitable scaling, $\Xi_n$ converges in $C_{\mathrm{loc}}^0(\mathbb{R}^2)$ to a solution $\hat\xi$ of the linearized problem
\begin{equation*}
\Delta\hat\xi+\dfrac{8\gamma (1+\alpha)^2|z|^{2\alpha}}{(1+\gamma|z|^{2(1+\alpha)})^2}\hat\xi=0~\mathrm{in}~\mathbb{R}^2 \quad \mathrm{and}\quad |\hat\xi(z)|\leq1~\mathrm{in}~\mathbb{R}^2,
\end{equation*}
where $\gamma=\frac{\pi\overline h_1(0)}{1+\alpha}$, see for example Lemma \ref{le4.1}. It follows from \cite[Corollary 2.2]{CLin4} that there exists a constant $b_0\in\mathbb{R}$ such that
\begin{equation} \label{uno}
\Xi_n(\sigma_n z)\to b_0 \frac{1-\gamma |z|^{2+2\alpha}}{1+\gamma |z|^{2+2\alpha}} \quad \mbox{in } C_{\mathrm{loc}}^0(\mathbb{R}^2).
\end{equation}

\medskip

\noindent \textbf{Step 2.} We next consider the global behavior of $\Xi_n$ away from the blow up point $p_i$. It follows from \eqref{2.1} that
$$c_n(x)\to0\quad\mathrm{in}\quad C_{\mathrm{loc}}^0(\overline\Omega\setminus\{0\}).$$	
Using then $\|\Xi_n\|_{L^\infty(\Omega)}\leq 1$ and \eqref{equa} it is not difficult to see that
$$
	\Xi_n\to\xi_0\quad\mbox{in } C_{\mathrm{loc}}^0(\overline\Omega\setminus\{0\}), \quad \Delta\xi_0=0\quad\mbox{in }\Omega. 
$$
Therefore $\xi_0=-b$ in $\Omega$ for some constant $b$ and
\begin{equation} \label{due}
\Xi_n\to-b\quad\mathrm{in}\quad C_{\mathrm{loc}}^0(\overline\Omega\setminus\{0\}).
\end{equation}
Finally, by an O.D.E. argument as in Lemma \ref{le4.2} one can show $b=b_0$.

\medskip

\noindent \textbf{Step 3.} We then study the asymptotic in the Pohozaev-type identity given by Lemma \ref{le5.1} (with suitable minor modifications, see for example \cite{bjly2}). Using the assumption $\nabla\mathcal{H}_{p_i}(p_i)=0$ it is possible to prove that 
$$
	b_0\ell(p_i)=o(1) \quad \mbox{for $n$ large},
$$
see section \ref{sec:poh}. Since by assumption $\ell(p_i)\neq0$ we deduce $b_0=0$.

\medskip

\noindent \textbf{Step 4.} The contradiction is then obtained by a blow up argument using $b=b_0=0$ jointly with \eqref{uno} and \eqref{due} exactly as in the proof of Theorem \ref{th.main}, see the end of section \ref{sec:poh}. The proof of Theorem \ref{th.main2} is completed.

\


\begin{thebibliography}{99}


\bibitem{B5} D. Bartolucci, {\em Global bifurcation analysis of mean field equations and
the Onsager microcanonical description of two-dimensional turbulence}, Calc. Var. P.D.E. (2019), 58:18.

\bibitem{bcct} D. Bartolucci, C.C. Chen, C.S. Lin, G. Tarantello,
{\em Profile of Blow Up Solutions To Mean Field Equations with Singular Data},
Comm. in P. D. E.  {\bf 29}(7-8) (2004), 1241-1265.

\bibitem{BDeM}
D. Bartolucci, F. De Marchis, {\em On the Ambjorn-Olesen electroweak condensates},
{Jour. Math. Phys.} {\bf 53} 073704 (2012).

\bibitem{BdM2} D. Bartolucci, F. De Marchis,
{\em Supercritical Mean Field Equations on convex domains and the Onsager's
statistical description of two-dimensional turbulence}, Archive for Rational Mechanics and Analysis, {\bf 217}/2 (2015), 525-570.

\bibitem{BdMM} D. Bartolucci, F. De Marchis, A. Malchiodi, {\em Supercritical conformal metrics on
surfaces with conical singularities}, Int. Math. Res. Not. 2011, (2011){\bf (24)}, 5625-5643.

 \bibitem{bghjy} D. Bartolucci, C. Gui, Y. Hu, A. Jevnikar, W. Yang, \emph{Mean field equation on torus: existence and uniqueness of evenly symmetric blow-up solutions}, Preprint (2019); arxiv:1902.06934.

\bibitem{bgjm} D. Bartolucci, C. Gui, A. Jevnikar, A. Moradifam, \emph{A singular Sphere Covering Inequality: uniqueness and symmetry of solutions to singular Liouville-type equations}, Math. Ann. (2018); DOI:10.1007/s00208-018-1761-1.

\bibitem{bjly} D. Bartolucci, A. Jevnikar, Y. Lee, W. Yang,
{\em Uniqueness of bubbling solutions of mean field equations},\\
{J. Math. Pures Appl.} {\bf 123} (2019), 78-126.


\bibitem{bjly2} D. Bartolucci, A. Jevnikar, Y. Lee, W. Yang,
\emph{Non degeneracy, Mean Field Equations and the Onsager theory of $2D$ turbulence}, Arch. Rat. Mech. Anal. (ARMA) \textbf{230}(1) (2018), 397-426.

\bibitem{bjly3} D. Bartolucci, A. Jevnikar, Y. Lee, W. Yang,
\emph{Local uniqueness of $m$-bubbling sequences for the Gel'fand equation}, Comm. P. D. E. (2019); DOI:10.1080/03605302.2019.1581801.

\bibitem{bjl} D. Bartolucci, A. Jevnikar, C.S. Lin, \emph{Non-degeneracy and uniqueness of solutions to singular mean field equations on bounded domains}, J. Diff. Eq. \textbf{266}(1) (2019), 716-741. 

\bibitem{bl} D. Bartolucci, C.S. Lin, {\em Uniqueness Results for Mean Field Equations with Singular Data},
Comm. in P. D. E. {\bf 34}(7) (2009), 676-702.

\bibitem{BLin3} D. Bartolucci, C.S. Lin, {\em Existence and uniqueness for
Mean Field Equations on multiply connected domains at the critical parameter},
{Math. Ann.} {\bf 359} (2014), 1-44.

\bibitem{BLT} D. Bartolucci, C.S. Lin, G. Tarantello, {\em Uniqueness and symmetry results for
solutions of a mean field equation on ${\mathbb{S}}^{2}$ via a new bubbling phenomenon},
{Comm. Pure Appl. Math.} {\bf 64}(12) (2011), 1677-1730.

\bibitem{BMal} D. Bartolucci, A. Malchiodi, {\em An improved geometric
inequality via vanishing moments, with applications to singular
Liouville equations}, {Comm. Math. Phys.} {\bf 322} (2013), 415-452.

\bibitem{BM3} D. Bartolucci, E. Montefusco, {\em Blow up analysis,
	existence and qualitative properties of solutions for the two
	dimensional Emden-Fowler equation with singular potential},
M$^{2}$.A.S. {\bf 30}(18) (2007), 2309-2327.

\bibitem{bt} D. Bartolucci, G. Tarantello, {\em Liouville type equations with
	singular data and their applications to periodic multivortices for the
	electroweak theory}, Comm. Math. Phys. {\bf 229} (2002), 3-47.

\bibitem{bt2} D. Bartolucci, G. Tarantello, {\em Asymptotic blow-up analysis for singular Liouville type equations with
	applications}, J. Diff. Eq. {\bf 262} (2017), 3887-3931.

\bibitem{bm}
H. Brezis, F. Merle,
{\em Uniform estimates and blow-up behaviour for
solutions of $-\Delta u = V(x)e^{u}$ in two dimensions},
{Comm. in P.D.E.,}  {\bf 16}(8,9) (1991), 1223-1253.


\bibitem{clmp2} E. Caglioti, P.L. Lions, C. Marchioro, M. Pulvirenti,
{\em A special class of stationary flows for two dimensional Euler equations: a
statistical mechanics description. II}, Comm. Math. Phys. {\bf 174} (1995),
229-260.

\bibitem{cama} A. Carlotto, A. Malchiodi, {\em
Weighted barycentric sets and singular Liouville equations on compact surfaces},
J. Funct. Anal. \textbf{262}(2) (2012), 409-450.

\bibitem{cLin14} C.C. Chai, C.S. Lin, C.L. Wang, {\em Mean field equations, hyperelliptic curves, and
modular forms: I},  Camb. J. Math. {\bf 3}(1-2) (2015),  127-274.

\bibitem{CCL} S.Y.A. Chang, C.C. Chen, C.S. Lin, {\em Extremal functions for a mean field equation in two dimension},
Lecture on Partial Differential Equations, New Stud. Adv. Math. {\bf 2} Int. Press, Somerville, MA, 2003, 61-93.

\bibitem{CKLin} {Z.J. Chen,} T.J. Kuo, C.S. Lin, {\em Hamiltonian system for the elliptic form of Painlev\'{e} VI equation},
J. Math. Pure App. {\bf 106}(3) (2016),  546-581.

\bibitem{cl1} C. C. Chen, C.S. Lin, {\em Sharp estimates for solutions of multi-bubbles in compact Riemann surface.}
{ Comm.  Pure Appl.  Math. } \textbf{55} (2002), 728-771.

\bibitem{cl2} C.C.  Chen,   C.S.  Lin, {\em Topological degree for a mean field equation on Riemann surfaces.}
{ Comm.  Pure Appl.  Math. } \textbf{56} (2003), 1667-1727.

\bibitem{CLin4} C.C. Chen , C.S. Lin, {\em Mean field equations of Liouville type with
	singular data: sharper estimates}, Discr. Cont. Dyn. Syt. \textbf{28}(3) (2010), 1237-1272.

\bibitem{cl4} C.C.  Chen,  C.S.  Lin, {\em Mean field equation of Liouville type with singular data: topological degree.}
{  Comm.  Pure Appl.  Math. }  \textbf{68}(6)  (2015), 887-947.

\bibitem{dem2} F. De Marchis, {\em Generic multiplicity for a scalar field equation on compact surfaces},
J. Funct. An. \textbf{259} (2010), 2165-2192.

\bibitem{DJLW} W. Ding, J. Jost, J. Li, G. Wang, {\em Existence results for
mean field equations},  Ann. Inst. H. Poincar\'e Anal. Non Lin\'eaire {\bf 16}
(1999), 653-666.

\bibitem{dj} Z. Djadli, {\em Existence result for the mean field problem
on Riemann surfaces of all genuses}, Comm. Contemp. Math.  \textbf{10}(2) (2008), 205-220.


\bibitem{EGP} P. Esposito, M. Grossi, A. Pistoia, {\em On the existence of blowing-up solutions
for a mean field equation}, Ann. Inst. H. Poincar\'e Anal. Non Lin\'eaire {\bf 22}(2) (2005), 227-257.

\bibitem{GM1} C. Gui, A. Moradifam, {\em The Sphere Covering Inequality and Its Applications}, Invent. Math. \textbf{214}(3) (2018), 1169-1204.

\bibitem{GM3}{ C. Gui, A. Moradifam,} {\em Uniqueness of solutions of mean field equations in $\mathbb{R}^2$}, Proc. Am. Math. Soc. \textbf{146} (2018), 1231-1242.

\bibitem{KW} J.L. Kazdan, F.W. Warner,
{\em Curvature functions for compact 2-manifolds}, Ann. Math. {\bf 99}  (1974), 14-74.

\bibitem{KMdP} M. Kowalczyk, M. Musso, M. del Pino,  {\em Singular limits in
Liouville-type equations}, Calc. Var. P.D.E. {\bf 24}(1)  (2005), 47-81.

\bibitem{KLin} T.J. Kuo, C.S. Lin, {\em Estimates of the mean field equations with integer singular sources: non-simple blow up},
Jour. Diff. Geom. {\bf 103} (2016), 377-424.

\bibitem{yy} Y.Y. Li,  {\em Harnack type inequality: the method of moving planes},
Comm. Math. Phys.,  {\bf 200} (1999), 421-444.

\bibitem{ls} Y.Y. Li, I. Shafrir, {\em Blow-up analysis for Solutions of $-\Delta u = V(x)e^{u}$
in dimension two}, {Ind. Univ. Math. J.},  {\bf 43}(4) (1994), 1255-1270.

\bibitem{Lin1} C.S. Lin, {\em Uniqueness of solutions to the mean field equation for the
spherical Onsager Vortex}, Arch. Rat. Mech. An. {\bf  153} (2000), 153-176.

\bibitem{Lin7} C.S. Lin, M. Lucia, {\em Uniqueness of solutions for a
mean field equation on torus},  J. Diff. Eq.  {\bf 229}(1)  (2006), 172-185.

\bibitem{linwang} C.S. Lin, C.L. Wang, {\em Elliptic functions, Green functions
and the mean field equations on tori}, Ann. of Math. {\bf 172}(2) (2010), 911-954.

\bibitem{ly} C.S. Lin, S. Yan, {\em On the mean field type bubbling solutions for Chern-Simons-Higgs equation},
{Advances in Mathematics} {\bf 338} (2018), 1141-1188.

\bibitem{MaW} L. Ma, J. Wei,
{\em Convergence for a Liouville equation}, Comment. Math. Helv. {\bf 76} (2001), 506-514.

\bibitem{pot} A. Poliakovsky, G. Tarantello, {\em On a planar Liouville-type problem in the study of
selfgravitating strings}, J. Diff. Eq. {\bf 252} (2012), 3668-3693.

\bibitem{PT} J.  Prajapat, G.  Tarantello, {\em On a class of elliptic problems in $\mathbb{R}^2$: Symmetry and uniqueness results}, { Proc.  Roy.  Soc.  Edinburgh
Sect.  A. } \textbf{131} (2001), 967-985.

\bibitem{sy2} J. Spruck, Y. Yang, {\em On Multivortices in the Electroweak Theory I:Existence of Periodic Solutions},
Comm. Math. Phys. {\bf 144} (1992), 1-16.

\bibitem{suz} T. Suzuki, {\em Global analysis for a two-dimensional elliptic eiqenvalue problem with the exponential
                nonlinearly}, Ann. Inst. H. Poincar\'e Anal. Non Lin\'eaire {\bf 9}(4) (1992), 367-398.

\bibitem{T0} G. Tarantello,
{\em Multiple condensate solutions for the Chern-Simons-Higgs theory},
{J. Math. Phys.} {\bf 37} (1996), 3769-3796.

\bibitem{Troy} M. Troyanov, {\em Prescribing curvature on compact surfaces with
conical singularities}, Trans. Amer. Math. Soc. {\bf 324} (1991), 793-821.

\bibitem{wz} J. Wei, L. Zhang, \emph{Estimates for Liouville equation with quantized singularities}, Preprint (2019); arxiv: 1905.04123. 

\bibitem{w} G. Wolansky, {\em On steady distributions of self-attracting
clusters under friction and fluctuations}, Arch. Rational Mech. An.
{\bf 119} (1992), 355-391.

\bibitem{yang} Y. Yang, "Solitons in Field Theory and Nonlinear Analysis",
Springer Monographs in Mathematics, Springer, New York, 2001.

\bibitem{Za2} L. Zhang, {\em Asymptotic behavior of blowup solutions for elliptic equations
	with exponential nonlinearity and singular data}, Commun. Contemp. Math. {\bf 11} (2009), 395-411.

\end{thebibliography}
\end{document}